\documentclass[reqno]{amsart}

\usepackage{amsfonts}
\usepackage{amsthm}
\usepackage{amstext}
\usepackage{amsmath}
\usepackage{amssymb}
\usepackage{paralist}

 \usepackage[colorlinks=true]{hyperref}
\hypersetup{urlcolor=blue, citecolor=red}

\usepackage{accents}  
\usepackage[pagewise,mathlines]{lineno}

\theoremstyle{plain}
\newtheorem{theorem}{Theorem}
\newtheorem{proposition}[theorem]{Proposition}
\newtheorem{corollary}[theorem]{Corollary}
\newtheorem{lemma}[theorem]{Lemma}

\newtheorem{theoremmain}{Theorem}
\newcommand{\refThmA}{\ref{T:intro-entropy-fk}}
\newcommand{\refThmB}{\ref{T:intro-entropy-graphs}}
\newcommand{\refThmC}{\ref{T:intro-strictlyErg}}

\theoremstyle{definition}

\newtheorem{example}[theorem]{Example}

\newtheorem{remark}[theorem]{Remark}

\numberwithin{equation}{section}

\newcommand{\NNN}{\mathbb N}

\newcommand{\RRR}{\mathbb R}

\newcommand{\AAa}{\mathcal{A}}
\newcommand{\BBb}{\mathcal{B}}
\newcommand{\FFf}{\mathcal{F}}

\newcommand{\VVv}{\mathcal{V}}
\newcommand{\PPp}{\mathcal{P}}

\newcommand{\ubar}[1]{\underaccent{\bar}{#1}}

\newcommand{\ccc}{c}
\newcommand{\cccu}{\bar{\ccc}}
\newcommand{\cccl}{\ubar{\ccc}}
\newcommand{\entru}{{\bar{h}_{\operatorname{cor}}}}
\newcommand{\entrl}{{\ubar{h}_{\operatorname{cor}}}}
\newcommand{\entr}{{{h}_{\operatorname{cor}}}}
\newcommand{\cdimu}{\bar{d}_{\operatorname{cor}}}
\newcommand{\cdiml}{\ubar{d}_{\operatorname{cor}}}

\newcommand{\boxdimu}{\bar{d}_{\operatorname{box}}}
\newcommand{\boxdiml}{\ubar{d}_{\operatorname{box}}}

\newcommand{\mmod}{\operatorname{mod}}
\newcommand{\diam}{\operatorname{diam}}
\newcommand{\dist}{\operatorname{dist}}
\newcommand{\orbit}{\operatorname{Orb}}
\newcommand{\topol}{{\operatorname{top}}}
\newcommand{\closure}[1]{\overline{#1}}
\newcommand{\eps}{\varepsilon}

\newcommand{\abs}[1]{\lvert#1\rvert}
\newcommand{\absbig}[1]{{\left|#1\right|}}
\newcommand{\len}[1]{\abs{#1}}
\newcommand{\card}[1]{{\abs{#1}}}
\newcommand{\cardbig}[1]{{\big|{#1}\big|}}

\newcommand{\cardw}[1]{\operatorname{card}#1}

\title[Local correlation entropy]{Local correlation entropy}

\author[Vladim\'ir \v Spitalsk\'y]{}



\email{vladimir.spitalsky@umb.sk}

\begin{document}
\maketitle

\centerline{\scshape Vladim\'ir \v Spitalsk\'y}
\medskip
{\footnotesize
   \centerline{Department of~Mathematics, Faculty of~Natural Sciences}
   \centerline{Matej Bel University}
   \centerline{Tajovsk\'eho~40, Bansk\'a Bystrica, Slovakia}
} 

\bigskip

\begin{abstract}
Local correlation entropy, introduced by Takens in 1983, represents the
exponential decay rate of the relative frequency of
recurrences in the trajectory of a point, as the embedding dimension grows to infinity.
In this paper we study relationship between the supremum of local correlation entropies and
the topological entropy.
For dynamical systems on topological graphs we prove that the two quantities coincide.
Moreover, there is an uncountable set of points with local correlation entropy arbitrarily close
to the topological entropy. On the other hand, we construct a strictly ergodic subshift
with positive topological entropy having all local correlation entropies equal to zero.
As a necessary tool, we derive
an expected relationship between the local correlation entropies
of a system and those of its iterates.
\end{abstract}


\section{Introduction}\label{S:intro}
A (topological) dynamical system is a pair $(X,f)$ where $X$ is a compact metric space $X$
and $f:X\to X$ is a continuous map. A point $x\in X$
is recurrent when its trajectory $(f^n(x))_{n=0}^\infty$
returns repeatedly to every neighborhood of $x$.
The topological version of the famous Poincar\'e recurrence theorem states that,
with respect to every invariant Borel measure, almost every point is recurrent.
So if we look at the trajectory
of a typical point $x$, we see infinitely many indices $n$ such that $f^n(x)$ is close to $x$.
Moreover, continuity of $f$ implies that we see infinitely many
pairs of indices $i\ne j$ such that $f^i(x)$ is close to $f^j(x)$. Such pairs
are called recurrences.

Recurrences can be effectively visualized via recurrence plots, introduced by
Eckmann, Kamphorst, and Ruelle in \cite{eckmann1987recurrence}. In its basic form, a recurrence
plot is a black-and-white square image with black pixels representing recurrences.
Quantitative study of patterns occurring in recurrence plots
is the subject of recurrence quantification analysis initiated by
Zbilut and Webber \cite{zbilut1992embeddings}; for surveys see
\cite{marwan2007recurrence,webber2015recurrence}.

In connection with correlation dimension
\cite{grassberger1983characterization,grassberger1983measuring} and correlation entropy
\cite{takens1983invariants} introduced in the beginning of 80's,
the so-called correlation sums were studied. Recall that the
\emph{correlation sum $C_\varrho(x,n,\eps)$}
of (the beginning of) the trajectory of a point $x$ is
\begin{equation}\label{EQ:intro-corsum(m=1)}
  C_{\varrho}(x,n,\eps)
  =
  \frac{1}{n^2}\,
  \cardw\big\{
    (i,j):\ 0\le i,j<n,\ \varrho(f^i(x),f^j(x))\le \eps
  \big\},
\end{equation}
where $\varrho$ is the metric of $X$, $n\in\NNN$, and $\eps>0$.
It is the relative frequency of recurrences seen in the initial segment of the trajectory of
$x$, with closeness defined by the metric $\varrho$ and the distance threshold $\eps$
(with pairs $(i,i)$ counted as recurrences).
Correlation sums appear naturally in different contexts.
They are used in the estimation of correlation
dimension and correlation entropy. In the recurrence quantification analysis, several of the
basic quantitative characteristics can be expressed in terms of
correlation sums \cite{grendar2013strong}.
Also note that, by removing the diagonal pairs $(i,i)$, correlation sum becomes a $U$-statistic
\cite{denker1986rigorous,aaronson1996strongLaws}.

One of the fundamental results
states that, with respect to any $f$-ergodic measure $\mu$,
correlation sums of $\mu$-almost every point $x$ converges to the \emph{correlation integral}
\begin{equation}\label{EQ:intro-corint(m=1)}
\begin{split}
  \ccc_{\varrho}(\mu,\eps)
  &= \mu\times\mu\, \big\{(y,z)\in X\times X:\ \varrho(y,z)\le \eps \big\}
  \\
  &= \int_X \mu B_{\varrho}(x,\eps)\, d\mu(x)
\end{split}
\end{equation}
where $B_{\varrho}(x,\eps)$ denotes the closed ball
with the center $x$ and radius $\eps$. This result, proved
(by different methods and under different conditions) in
\cite{pesin1993rigorous,pesinTempelman1995,aaronson1996strongLaws,serinko1996ergodic,manningSimon1998},
justifies the use of correlation sums in estimating the correlation dimension,
as suggested by \cite{grassberger1983characterization,grassberger1983measuring}.

The correlation entropy, introduced by Takens \cite{takens1983invariants},
is a quantitative characteristic based on correlation sums / integrals.
To define it, in (\ref{EQ:intro-corsum(m=1)}) and (\ref{EQ:intro-corint(m=1)})
replace the metric $\varrho$ by Bowen's one
\begin{equation}\label{EQ:def-rhom}
  \varrho^f_m(y,z) = \max_{0\le i<m} \varrho(f^i(y),f^i(z))
  \qquad
  (y,z\in X).
\end{equation}
The obtained quantities are the correlation sum $C_m^f(x,n,\eps)$ and the correlation integral
$\ccc_m^f(\mu,\eps)$ corresponding to the trajectory of $x$ embedded to $X^m$.
The \emph{upper} and \emph{lower correlation entropies of an $f$-invariant measure $\mu$}
\cite[p.~361]{broer2010handbook}
quantify exponential decay rate of correlation integrals as $m$ grows to infinity
\begin{equation}\label{EQ:def-corel-entropy-metric}
\begin{split}
  &\entru(f,\mu) = \lim_{\eps\to 0} \limsup_{m\to\infty} (-1/m)\log \ccc^f_m(\mu,\eps),
\\
  &\entrl(f,\mu) = \lim_{\eps\to 0} \liminf_{m\to\infty} (-1/m)\log \ccc^f_m(\mu,\eps).
\end{split}
\end{equation}
Correlation entropy is a member of a $1$-parameter family of entropies
\cite{takens1998generalized,takens1999multifractal}.

The definition above which is recently used in the literature,
differs from the original one \cite{takens1983invariants}
by using correlation integrals instead of correlation sums.
Consequently, it depends on an invariant measure $\mu$ instead of a point $x$.
To distinguish the original definition from the recently used one, the
correlation entropy of $f$ at a point $x$ will be called local.
So, following \cite{takens1983invariants},
the \emph{upper} and \emph{lower local correlation entropies of $f$ at $x$} are defined by
\begin{equation}\label{EQ:def-corel-entropy-topol}
\begin{split}
  &\entru(f,x) = \lim_{\eps\to 0} \limsup_{m\to\infty} (-1/m)\log \cccl_m^f(x,\eps),
  \\
  &\entrl(f,x) = \lim_{\eps\to 0} \liminf_{m\to\infty} (-1/m)\log \cccu_m^f(x,\eps),
\end{split}
\end{equation}
where
\begin{equation}\label{EQ:c(x)-def}
  \cccu_m^f(x,\eps) = \limsup_{n\to\infty} C_m^f(x,n,\eps)\,,
  \qquad
  \cccl_m^f(x,\eps) = \liminf_{n\to\infty} C_m^f(x,n,\eps)\,.
\end{equation}
(Note that, in \cite{takens1983invariants}, the author considered the lower entropy only.)
Of course, due to the convergence of correlation sums to the correlation integral,
these local correlation entropies are often equal to the correlation entropy of a measure $\mu$.
Nevertheless, we believe that these local correlation entropies deserve to be studied,
for what we have several reasons.
First, the ergodic results hold (usually) only for almost every point, but,
from the topological point of view, local correlation entropy at every point
should be considered. Second, since local correlation entropy
depends solely on the trajectory of a selected point,
it is computationally more tractable than correlation entropy of a measure.
In fact, when estimating correlation entropy of
an invariant measure $\mu$, correlation sums are often used and thus
the local correlation entropy is being estimated; see e.g.~\cite[\textsection{}7.7]{broer2010handbook}.
Finally,
study of local correlation entropies can yield new results,
which have not yet been obtained for correlation entropy of a measure.

Let us now briefly outline the main results of this paper.
We start with summarizing basic properties of the local correlation entropy.
One of them is the relationship between
local correlation entropies of $f$ and those of its iterates $f^k$.
Since we were not able to find a corresponding result in the literature,
we have included a proof of it in this paper. The proof is based on a combinatorial lemma
(see \textsection\ref{SS:f^k-combinatorial-lemma}), which
gives a relationship between correlation sum of $f$ at a point $x$
and correlation sums of $f^k$ at points $f^h(x)$ ($0\le h <k$),
see Lemma~\ref{L:fk-graph-corollary}.

\begin{theoremmain}\label{T:intro-entropy-fk}
  Let $(X,f)$ be a dynamical system. Then, for every $k\in\NNN$ and $x\in X$,
  \begin{linenomath}\begin{equation*}
    \entru(f^k,x) = k\cdot \entru(f,x),
    \qquad
    \entrl(f^k,x) = k\cdot \entrl(f,x).
  \end{equation*}\end{linenomath}
\end{theoremmain}

The basic motivation of the paper comes from studying
the relationship between the local correlation entropies
and the topological entropy of the system $(X,f)$. Already Takens \cite{takens1983invariants}
proved that the lower local correlation entropy is bounded from above by the topological entropy of
$f$ restricted to the orbit closure of $x$.
In Proposition~\ref{P:correl-and-topol-entropy-upper-bound}
we prove that this is true also for the upper local correlation entropy, which yields that
\begin{linenomath}\begin{equation*}
  \sup_{x\in X} \entrl(f,x)
  \le
  \sup_{x\in X} \entru(f,x)
  \le
  h_{\topol}(f).
\end{equation*}\end{linenomath}
We will show that, for dynamical systems on topological graphs,
the above inequalities are in fact equalities.
Recall that a topological graph is a continuum which can be
written as the union of finitely many arcs any two of which are either
disjoint or intersect only in one or both of their end points.

\begin{theoremmain}\label{T:intro-entropy-graphs}
  Let $X$ be a topological graph and $f:X\to X$ be a continuous map. Then
  \begin{linenomath}\begin{equation*}
    \sup_{x\in X} \entrl(f,x)
    = \sup_{x\in X} \entru(f,x)
    = h_\topol(f).
  \end{equation*}\end{linenomath}
  Moreover, for every $h < h_\topol(f)$ there is a Cantor set $X_h\subseteq X$
  such that $\entrl(f,x) \ge h$ for every $x\in X_h$.
\end{theoremmain}

The conclusion of Theorem~\refThmB{} clearly also holds for any (uncountable) system
with zero topological
entropy, and for any full shift (see Corollary~\ref{C:subshifts-correl-entropy-bernoulli}).
However, for general dynamical systems the supremum of local correlation entropies can
be strictly smaller than the topological entropy. We prove this by constructing
a strictly ergodic subshift with positive entropy and with all local correlation
entropies equal to zero; our construction
is a modification of Grillenberger's one \cite{grillenberger1973constructions}.

\begin{theoremmain}\label{T:intro-strictlyErg}
  There is a subshift $(X,\sigma)$ such that
  \begin{enumerate}
    \item[(a)] $(X,\sigma)$ is strictly ergodic;
    \item[(b)] $(X,\sigma)$ has positive topological entropy;
    \item[(c)] the local correlation entropy $\entr(\sigma,y)$ at every $y\in X$ is zero;
    \item[(d)] the correlation entropy $\entr(\sigma,\mu)$ of the unique invariant measure
               $\mu$ is zero.
  \end{enumerate}
\end{theoremmain}

For some other results which are worth mentioning and are not covered by
Theorems A--C, see Corollary~\ref{C:entr(fhx)=entr(x)} and
Propositions~\ref{P:strong-law-uniquelyErgodic} and \ref{P:extension-of-subshift}.

The paper is organized as follows. In \textsection\ref{S:prelim} we recall definitions and
known facts which will be required later. In \textsection\textsection\ref{S:proofA}
and \ref{S:proofB} we prove Theorems~\refThmA{} and \refThmB{}. A technical lemma
concerning strictly ergodic subshifts is given in \textsection\ref{S:uniqueErg}.
Finally, in \textsection\ref{S:strictlyErgExample} we prove Theorem~\refThmC{}.

\vspace{-8pt}
\section{Preliminaries}\label{S:prelim}
We write $\NNN$ ($\NNN_0$) for the set of positive (nonnegative) integers.
If no confusion can arise, segments
of integers $\{n,n+1,\dots,m-1\}$ ($n<m$) will be denoted by $[n,m)$.
For $x\in\RRR$, $\lceil x \rceil$ and $\lfloor x \rfloor$
denotes the ceiling and the floor of $x$, that is, the smallest
integer greater than or equal to $x$, and
the largest integer smaller than or equal to $x$.
The cardinality of a set $A$ is denoted by $\card{A}$
or by $\cardw{A}$. By $\log$ we mean the natural logarithm.

Let $X=(X,\varrho)$ be a metric space and $A$ be a subset of it.
The diameter of a subset $A$ of $X$ is denoted by $\diam_\varrho(A)$.
By $B_\varrho(x,\eps)$ we mean the closed ball with the center $x$ and radius $\eps$,
and by $B_\varrho(A,\eps)$ we mean the union of all $B_\varrho(x,\eps)$ with $x\in A$.
The set $A$ is called \emph{$\eps$-separated} if $\varrho(x,y)>\eps$ for every $x\ne y$ from $A$.
It is said to \emph{$\eps$-span} $X$ if $B_\varrho(A,\eps)=X$.
The smallest cardinality of an $\eps$-spanning subset of $X$ is denoted by $r_\varrho(\eps,X)$,
and the largest cardinality of an $\eps$-separated subset of $X$ is denoted by $s_\varrho(\eps,X)$.
If $X$ is compact, both $r_\varrho(\eps,X)$ and $s_\varrho(\eps,X)$ are always finite,
and we can define
the \emph{upper} and \emph{lower box dimensions} of $X$ by \cite[\textsection{}2.1]{falconer2014fractal}
\begin{linenomath}\begin{equation*}
  \bar{d}_{\operatorname{box}}(X;\varrho)
    = \limsup_{r\to 0} \frac{\log r_\varrho(\eps,X)}{-\log\eps}
  \quad\text{and}\quad
  \ubar{d}_{\operatorname{box}}(X;\varrho)
    = \liminf_{r\to 0} \frac{\log r_\varrho(\eps,X)}{-\log\eps} \,.
\end{equation*}\end{linenomath}


A \emph{measure-theoretical dynamical system} is a quadruple $(X,\FFf,\mu,f)$,
where $X$ is a nonempty set, $\FFf$ is a $\sigma$-algebra of subsets of $X$,
$\mu$ is a probability measure on $(X,\FFf)$, and
$f:X\to X$ is an $\FFf$-measurable map preserving $\mu$ (that is,
$\mu\big(f^{-1}(A)\big) = \mu(A)$ for every $A\in\FFf$). The system
$(X,\FFf,\mu,f)$ is called \emph{ergodic} if $\mu(A)\in\{0,1\}$
for every $A\in\FFf$ such that $f^{-1}(A)=A$.

A \emph{(topological) dynamical system} is a pair $(X,f)$ where $X=(X,\varrho)$ is
a compact metric space and $f:X\to X$ is a continuous map. A set $A\subseteq X$
is said to be \emph{$f$-invariant} if $f(A)\subseteq A$. A system $(X,f)$ is
\emph{minimal} if there is no nonempty proper closed $f$-invariant subset of $X$.
Every point of a minimal system $(X,f)$ is \emph{almost periodic}:
for every neighborhood $U$ of $x$ the return time set $N(x,U)$ is syndetic
(that is, it has bounded gaps).

An \emph{$f$-invariant measure} of $(X,f)$ is any Borel probability measure $\mu$ such that
$(X,\BBb,\mu,f)$, with $\BBb$ denoting the Borel $\sigma$-algebra on $X$,
is a measure-theoretical dynamical system. If $(X,\BBb,\mu,f)$ is ergodic
we say that $\mu$ is \emph{$f$-ergodic}. A system $(X,f)$ is called \emph{uniquely ergodic}
if it has unique invariant measure; if it is also minimal it is called \emph{strictly ergodic}.

Let $(X,f)$ be a (topological) dynamical system and
$\varrho$ be the metric of $X$.
For $m\in\NNN$ define (equivalent) Bowen's metric $\varrho^f_m$ on $X$ as in Introduction.
We write $B_m^f(x,\eps)$, $r^f_m(\eps,K)$, and $s^f_m(\eps,K)$
instead of $B_{\varrho_m^f}(x,\eps)$, $r_{\varrho_m^f}(\eps,X)$, and $s_{\varrho_m^f}(\eps,X)$.
A subset $A$ of $X$ is called \emph{$(m,\eps)$-spanning} or \emph{$(m,\eps)$-separated}
if it is $\eps$-spanning or $\eps$-separated with respect to $\varrho_m^f$.
By Bowen's definition of the \emph{topological entropy},
\begin{linenomath}\begin{equation*}
 h_\topol{}(f)
 = \lim_{\eps\to 0} \limsup_{m\to\infty} (1/m) \log r^f_m(\eps,X)
 = \lim_{\eps\to 0} \limsup_{m\to\infty} (1/m) \log s^f_m(\eps,X) \,.
\end{equation*}\end{linenomath}

\subsection{Local correlation entropy}
Let $X=(X,\varrho)$ be a compact metric space with a metric $\varrho$, and $f:X\to X$
be a continuous map.
For $m\in\NNN$, $x\in X$, $\eps>0$, and $n\in\NNN$ define
the \emph{correlation sum} $C^f_m(x,n,\eps)$
by
\begin{linenomath}\begin{equation*}
  C^f_m(x,n,\eps) = \frac{1}{n^2} \,
  \cardw\big\{
    (i,j):\ 0\le i,j<n,\ \varrho^f_m(f^i(x),f^j(x))\le \eps
  \big\}.
\end{equation*}\end{linenomath}
Recall the definition (\ref{EQ:def-corel-entropy-topol}) of
the \emph{upper} and \emph{lower local correlation entropies}
$\entru(f,x)$ and $\entrl(f,x)$ of $f$ at $x$.
If $\entru(f,x)=\entrl(f,x)$ then we say that the local correlation entropy
$\entr(f,x)$ of $f$ at $x$
exists and we put $\entr(f,x)=\entru(f,x)=\entrl(f,x)$.
If $\mu$ is an $f$-invariant probability,
the \emph{upper} and \emph{lower (measure-theoretic) correlation entropies} (of order $2$)
of $f$ with respect to $\mu$
are defined by (\ref{EQ:def-corel-entropy-metric}),
see e.g.~\cite[p.~361]{broer2010handbook}.
Notice that in this paper we deal solely with
correlation entropies of order $q=2$; for the definition and properties of
(measure-theoretic) correlation entropies of arbitrary order $q$
see e.g.~\cite{takens1998generalized,verbitskiy2000phdthesis,broer2010handbook}.

In the following we summarize some of the known results which will be used later.
The first one was in fact proved in \cite[p.~355]{takens1983invariants}, see also
\cite[Lemma~2.14]{verbitskiy2000phdthesis}.
\begin{proposition}[\cite{takens1983invariants}] \label{P:takens-corrEntropy-measureEntropy}
  Let $(X,f)$ be a dynamical system. Then, for every $f$-invariant measure $\mu$,
  \begin{linenomath}\begin{equation*}
    \entru(f,\mu)\le h_\mu(f).
  \end{equation*}\end{linenomath}
\end{proposition}
Correlation entropy $\entru(f,\mu)$ can be strictly smaller than measure-theoretic entropy.
For example, in \cite[Example~2.28]{verbitskiy2000phdthesis} the author constructs
a subshift $(X,\sigma)$ with invariant measure $\mu$ such that $\entr(f,\mu)=0$ and
$h_\mu(f) >0$.

The following result was first proved by Pesin \cite{pesin1993rigorous}, see also
\cite{pesinTempelman1995,aaronson1996strongLaws,serinko1996ergodic,manningSimon1998}.
(There, the space $X$ can be any complete separable metric space.)
\begin{proposition}[\cite{pesin1993rigorous}] \label{P:strong-law}
  Let $(X,f)$ be a dynamical system. Then, for every $f$-ergodic measure $\mu$,
  \begin{linenomath}\begin{equation*}
    \cccl^f_m(x,\eps)=\cccu^f_m(x,\eps)=\ccc^f_m(\mu,\eps)
  \end{equation*}\end{linenomath}
  for $\mu$-a.e.~$x\in X$ and every $\eps>0$ which is a continuity point of $\ccc^f_m(\mu,\cdot)$.
\end{proposition}
As a consequence of Proposition~\ref{P:strong-law} we obtain that, for ergodic $\mu$,
\begin{equation}\label{EQ:entr-strong-law}
    \entru(f,x) = \entru(f,\mu)
    \quad\text{and}\quad
    \entrl(f,x) = \entrl(f,\mu)
\end{equation}
for $\mu$-a.e.~$x\in X$.
For uniquely ergodic systems one can strengthen the previous theorem and obtain convergence of
correlation sums to correlation integral for every point.
\begin{proposition}\label{P:strong-law-uniquelyErgodic}
  Let $(X,f)$ be a uniquely ergodic dynamical system and $\mu$ be the unique $f$-invariant measure.
  Then
  \begin{linenomath}\begin{equation*}
    \cccl^f_m(x,\eps)=\cccu^f_m(x,\eps)=\ccc^f_m(\mu,\eps)
  \end{equation*}\end{linenomath}
  for {every} $x\in X$ and every $\eps>0$ which is a continuity point of $\ccc^f_m(\mu,\cdot)$.
\end{proposition}
\begin{proof}
  For any $y\in X$, the Dirac measure at $y$ is denoted by $\delta_y$. Fix $x\in X$, $m\in\NNN$, 
  and $\eps>0$. Unique ergodicity of $(X,f)$ implies that measures 
  $\mu_n=(1/n) \sum_{i=0}^{n-1} \delta_{f^i(x)}$ converge to $\mu$ in the weak*-topology
  (see e.g.~\cite[p.~106]{einsiedler2011ergodic}). 
  Thus $\mu_n\times\mu_n \to \mu\times\mu$ \cite[Lemma~1.1, p.~57]{parthasarathy1967probability}.
  The set $B=\{(x,y)\in X\times X:\ \varrho_m^f(x,y)\le\eps\}$ is closed
  and $C^f_m(x,n,\eps)=\mu_n\times\mu_n(B)$ for every $n$, so
  \begin{linenomath}\begin{equation*}
	\cccu^f_m(x,\eps)
	= \limsup_{n\to\infty} \mu_n\times\mu_n(B)
	\le \mu\times\mu(B)
	=\ccc^f_m(\mu,\eps)
  \end{equation*}\end{linenomath}
  (see e.g.~\cite[Theorem~6.1, p.~40]{parthasarathy1967probability}).
  On the other hand, the set $B^o=\{(x,y)\in X\times X:\ \varrho_m^f(x,y)<\eps\}\subseteq B$
  is open and so
  \begin{linenomath}\begin{equation*}
	\cccl^f_m(x,\eps)
	\ge \liminf_{n\to\infty} \mu_n\times\mu_n(B^o)
	\ge \mu\times\mu(B^o)
	=\lim_{\eps'\nearrow \eps}\ccc^f_m(\mu,\eps').
  \end{equation*}\end{linenomath}  
  Hence $\cccl^f_m(x,\eps)=\cccu^f_m(x,\eps)=\ccc^f_m(\mu,\eps)$
  provided $\ccc^f_m(\mu,\cdot)$ is continuous at $\eps$.
\end{proof}

\subsection{Correlation dimension}
Correlation dimension \cite{grassberger1983characterization,grassberger1983measuring}
is another widely used
characteristic based on the correlation integral.
Recall that upper and lower correlation dimensions (of order $2$) of a measure $\mu$
are defined by
\begin{equation}\label{EQ:def-corel-dim-metric}
  \cdimu(\mu) = \limsup_{\eps\to 0} \frac{\log \ccc_\varrho(\mu,\eps)}{-\log\eps},
  \qquad
  \cdiml(\mu) = \liminf_{\eps\to 0} \frac{\log \ccc_\varrho(\mu,\eps)}{-\log\eps} \,.
\end{equation}
One can analogously define upper and lower \emph{local correlation dimensions}
$\cdimu(f,x)$ and $\cdiml(f,x)$ by
\begin{equation}\label{EQ:def-corel-dim-topol}
  \cdimu(f,x) = \limsup_{\eps\to 0} \frac{\log \cccl^f_1(x,\eps)}{-\log\eps},
  \qquad
  \cdiml(f,x) = \liminf_{\eps\to 0} \frac{\log \cccu^f_1(x,\eps)}{-\log\eps} \,.
\end{equation}

\subsection{Shifts and subshifts}\label{SS:shifts}
Let $p\ge 2$ be an integer and $\AAa_p=\{0,1,\dots,p-1\}$. Put
\begin{linenomath}\begin{equation*}
 \Sigma_p = \AAa_p^{\NNN_0}=\{x=(x_i)_{i=0}^\infty:\ x_i\in \AAa_p \text{ for every } i\}.
\end{equation*}\end{linenomath}
Define a metric $\varrho$ on $\Sigma_p$ by
\begin{linenomath}\begin{equation*}
 \varrho(x,y) = 2^{-k},\qquad
 k=\min\{i\ge 0:\ x_i\ne y_i\}
\end{equation*}\end{linenomath}
for $x\ne y$, and $\varrho(x,y)=0$ for $x=y$; thus $\varrho(x,y)\le \frac12$ if and only if $x_0=y_0$.
Then $(\Sigma_p,\varrho)$ is a compact metric space homeomorphic to the Cantor ternary set.
The \emph{shift} $\sigma:\Sigma_p\to\Sigma_p$ is defined by
\begin{linenomath}\begin{equation*}
 \sigma((x_i)_i) = (y_i)_i,\qquad\text{where } y_i = x_{i+1}
 \text{ for every }i.
\end{equation*}\end{linenomath}
The dynamical system $(\Sigma_p,\sigma)$ is called the \emph{(one-sided) full shift}
on $p$ symbols.
If $X\subseteq \Sigma_p$ is a nonempty closed $\sigma$-invariant  set
then the restriction $\sigma|_X:X\to X$ is called a \emph{subshift};
since no confusion can arise, the restriction $\sigma|_X$ will be denoted by $\sigma$.

The members of $\AAa_p^* = \bigcup_{k\ge 0} \AAa_p^k$ are called \emph{words}.
Let $k\ge 0$ and $w=w_0\dots w_{k-1}\in\AAa_p^k$. Then we say
that $w$ is a \emph{$k$-word} and that the \emph{length} of it is $\abs{w}=k$.
The \emph{cylinder $[w]$} is the clopen set
$\{x\in\Sigma_p:\ x_i=w_i \text{ for every }0\le i<k\}$.

For a $\sigma$-invariant measure $\mu$ put
\begin{equation}\label{EQ:subshifts-muk-def}
    \tilde\mu{(k)} = \sum_{w\in\AAa_p^k} \big(\mu([w])\big)^2.
\end{equation}
The next two lemmas (for the second one see e.g.~\cite[p.~774]{takens1998generalized})
follows from the fact that $\varrho^\sigma_m(y,z)\le 2^{-k}$
if and only if $\varrho(y,z)\le 2^{-(k+m-1)}$ if and only if
there is $w\in\AAa_p^{k+m-1}$ such that $y,z\in[w]$.

\begin{lemma}\label{L:subshifts-correl-entropy-local}
  Let $(X,\sigma)$ be a subshift and $\eps\in(0,1]$.
  Let $k\ge 0$ be an integer such that $\eps\in [2^{-k},2^{-(k-1)})$.
  Then, for every $x\in X$ and $m,n\in\NNN$,
  \begin{linenomath}\begin{equation*}
    C_m^\sigma(x,n,\eps) = C_1^\sigma\big(x,n,2^{-(k+m-1)}\big).
  \end{equation*}\end{linenomath}
  Consequently,
  \begin{linenomath}\begin{equation*}
    \cccu_m^\sigma(x,\eps) = \cccu_1^\sigma\big(x,2^{-(k+m-1)}\big)
    \qquad\text{and}\qquad
    \cccl_m^\sigma(x,\eps) = \cccl_1^\sigma\big(x,2^{-(k+m-1)}\big).
  \end{equation*}\end{linenomath}
\end{lemma}

\begin{lemma}\label{L:subshifts-correl-entropy-measure}
  Let $(X,\sigma)$ be a subshift, $\mu$ be a $\sigma$-invariant measure, and $\eps\in(0,1]$.
  Let $k\ge 0$ be an integer such that $\eps\in [2^{-k},2^{-(k-1)})$.
  Then, for every $m\in\NNN$,
  \begin{linenomath}\begin{equation*}
    \ccc^\sigma_m(\mu,\eps)
    = \ccc^\sigma_1(\mu,2^{-(k+m-1)})
    = \tilde\mu{(k+m-1)},
  \end{equation*}\end{linenomath}
  and so
  \begin{linenomath}\begin{equation*}
  \begin{split}
    &\entru(\sigma,\mu) = \limsup_{m\to\infty}  (-1/m)\log \tilde\mu(m),
    \\[2mm]
    &\entrl(\sigma,\mu) = \liminf_{m\to\infty}  (-1/m)\log \tilde\mu(m).
  \end{split}
  \end{equation*}\end{linenomath}
\end{lemma}

If $\pi=(\pi_0,\dots,\pi_{p-1})$ is a probability vector (that is, $\pi_i\ge 0$
and $\sum_i\pi_i=1$),
then the ($\sigma$-invariant Borel probability) measure $\mu$ on $(\Sigma_p, \BBb(\Sigma_p))$
such that $\mu([w])=\prod_{i<k} \pi_{w_i}$ for every $k\ge 1$ and $w\in\AAa_p^k$, is called
the \emph{Bernoulli measure} generated by $\pi$.
An easy consequence of Lemma~\ref{L:subshifts-correl-entropy-measure} is the following
result, see \cite[p.~773]{takens1998generalized}, \cite[Sect.~2.5.2]{verbitskiy2000phdthesis}.

\begin{lemma}\label{L:subshifts-correl-entropy-bernoulli}
  Let $(\Sigma_p,\sigma)$ be the full shift, $\pi=(\pi_0,\dots,\pi_{p-1})$ be a probability vector,
  and $\mu$ be the Bernoulli measure generated by $\pi$.
  Then
\vspace*{3pt}  \begin{linenomath}\begin{equation*}
    \entr(\sigma,\mu) = -\log\left(\sum_{i<p} \pi_i^2\right).
  \end{equation*}\end{linenomath}
\end{lemma}

\begin{corollary}\label{C:subshifts-correl-entropy-bernoulli}
  Let $p\ge 2$ and let $(\Sigma_p,\sigma)$ be the full shift. Then for every $h\in[0,\log p]$
  there is a Cantor subset $X_h$ of $\Sigma_p$ such that
 \vspace*{3pt} \begin{linenomath}\begin{equation*}
    \entr(\sigma,x)=h
    \qquad\text{for every } x\in X_h.
  \end{equation*}\end{linenomath}
\end{corollary}

\begin{proof}
  Since $h\in[0,\log p]$, there is a probability vector $\pi=(\pi_0,\dots,\pi_{p-1})$ such that
  $\sum_i \pi_i^2=e^{-h}$. Let $\mu$ be the Bernoulli measure generated by $\pi$;
  note that $\mu$ is $\sigma$-ergodic.
  By (\ref{EQ:entr-strong-law}) and Lemma~\ref{L:subshifts-correl-entropy-bernoulli},
  there is a Borel subset $Y_h$ of $\Sigma_p$ such that $\mu(Y_h)=1$ and $\entr(\sigma,x)=h$
  for every $x\in Y_h$.
  Since $\mu$ is non-atomic, $Y_h$ is uncountable and hence it contains a Cantor set
  (see e.g.~\cite[Theorem~3.2.7]{srivastava2008course}).
\end{proof}

\section{Proof of Theorem~\refThmA}\label{S:proofA}

\begin{lemma}\label{L:c(x)>=eta^m}
  Let $X$ be a compact metric space and $\eps>0$. Put $\eta=r(\eps/2,X)^{-1}$. Then for every
  continuous map $f:X\to X$, $x\in X$, and $m,n\in\NNN$,
 \vspace*{3pt} \begin{linenomath}\begin{equation*}
    C^f_m(x,n,\eps) \ge \eta^m.
  \end{equation*}\end{linenomath}
  Consequently, $\cccu^f_m(x,\eps) \ge \cccl^f_m(x,\eps) \ge \eta^m$
  and
 \vspace*{3pt} \begin{linenomath}\begin{equation*}
    \cdimu(f,x)\le \boxdimu(X),
    \qquad
    \cdiml(f,x)\le \boxdiml(X).
  \end{equation*}\end{linenomath}
\end{lemma}

\begin{proof}
  Put $p=r(\eps/2,X)$, $\eta=1/p$,
  and take a finite subset $\{y_0,\dots,y_{p-1}\}$ of $X$ which $(\eps/2)$-spans $X$.
  Fix arbitrary continuous $f:X\to X$, $x\in X$, and $m,n\in\NNN$; for $i\ge 0$ denote $f^i(x)$
  by $x_i$.

  Recall that $\AAa_p^m$ is the set of $m$-words $w=w_0\dots w_{m-1}$ over
  $\AAa_p=\{0,\dots,p-1\}$. Take a partition
  $(N_w)_{w\in\AAa_p^m}$ of $\{0,1,\dots,n-1\}$ such that, for every $w=w_0\dots w_{m-1}$,
 \begin{linenomath}\begin{equation*}
    N_w \subseteq \{
      0\le i<n-1:\ x_{i+h}\in B(y_{w_h},\eps/2) \text{ for every } 0\le h<m
    \}.
  \end{equation*}\end{linenomath}
  Notice that $\varrho^f_m(x_i,x_j)\le\eps$ for every $i,j\in N_w$.
  Put $n_w=\card{N_w}$. Since $\sum_w n_w = n$, the arithmetic-quadratic mean inequality yields
  \begin{linenomath}\begin{equation*}
    C^f_m(x,n,\eps)
    \ge   \frac{1}{n^2}   \cdot \sum_{w\in\AAa_p^m} n_w^2
    \ge   \frac{1}{n^2}   \cdot \frac{n^2}{p^m}
    = \eta^m.
  \end{equation*}\end{linenomath}
\end{proof}

The easy proof of the following lemma is skipped.
\begin{lemma}\label{L:c(x)-props}
  Let $(X,f)$ be a dynamical system, $x\in X$, and $m\in\NNN$.
  Then
  \begin{enumerate}
    \item[(a)] $\cccu^f_m(x,\eps)$ and $\cccl^f_m(x,\eps)$ are non-decreasing functions
       of $\eps$ and non-increasing functions of $m$;
    \item[(b)] $0 < \cccl^f_m(x,\eps)\le \cccu^f_m(x,\eps)\le 1$ for every $\eps>0$;
    \item[(c)] $\cccl^f_m(x,\eps)= \cccu^f_m(x,\eps)= 1$ for every $\eps\ge\diam_{\varrho}(X)$.
  \end{enumerate}
\end{lemma}

The next lemma states that in the limits from (\ref{EQ:def-corel-entropy-topol})
and (\ref{EQ:c(x)-def})
one can use any sublacunary sequences $(n_j)_{j\ge 1}$ and $(m_j)_{j\ge 1}$ of integers.

\begin{lemma}\label{L:c(x)-via-subsequence}
  Let $(X,f)$ be a dynamical system, $m\in\NNN$, $\eps>0$, and $x\in X$. Let
  $(n_j)_j$, $(m_j)_j$ be increasing sequences of integers such that
  $n_{j+1}/n_j\to 1$ and $m_{j+1}/m_j\to 1$ for $j\to\infty$. Then
  \begin{linenomath}\begin{equation*}
    \cccu^f_m(x,\eps) = \limsup_{j\to\infty} C^f_m(x,n_j,\eps)\,,
    \qquad
    \cccl^f_m(x,\eps) = \liminf_{j\to\infty} C^f_m(x,n_j,\eps)\,,
  \end{equation*}\end{linenomath}
  and
  \begin{linenomath}\begin{equation*}
  \begin{split}
    &\entru(f,x) = \lim_{\eps\to 0} \limsup_{j\to\infty} (-1/m_j)\log\cccl^f_{m_j}(x,\eps)\,,
    \\
    &\entrl(f,x) = \lim_{\eps\to 0}\liminf_{j\to\infty} (-1/m_j)\log\cccu^f_{m_j}(x,\eps)\,.
  \end{split}
  \end{equation*}\end{linenomath}
\end{lemma}

\begin{proof}
  If $n_j\le n<n_{j+1}$ then
  \begin{linenomath}\begin{equation*}
    \left( \frac{n_j}{n}  \right)^2  C^f_m(x,n_j,\eps)
    \le
    C^f_m(x,n,\eps)
    \le
    \left( \frac{n_j}{n}  \right)^2  C^f_m(x,n_j,\eps) + \frac{n^2-n_j^2}{n^2}\,.
  \end{equation*}\end{linenomath}
  Since correlation sums are bounded,
  $\abs{C^f_m(x,n,\eps)-C^f_m(x,n_j,\eps)}$ is arbitrarily small
  for $j$ large enough.
  Now the first part of the lemma follows.

  For $m\in\NNN$ put $a_m=-\log\cccl^f_{m}(x,\eps)$. By Lemma~\ref{L:c(x)-props},
  $0\le a_m \le a_{m+1}$ for every $m$. Thus
  \begin{linenomath}\begin{equation*}
    \frac{m_{j}}{m_{j+1}} \cdot \frac{a_{m_{j}}}{m_{j}}
    \le
    \frac{a_m}{m}
    \le
    \frac{m_{j+1}}{m_j} \cdot \frac{a_{m_{j+1}}}{m_{j+1}}
  \end{equation*}\end{linenomath}
  whenever $m_j\le m<m_{j+1}$.
  Using this and the fact that $a_m/m\le r(\eps/2,X)$ for every $m$ by Lemma~\ref{L:c(x)>=eta^m},
  we easily obtain that
  \begin{linenomath}\begin{equation*}
    \limsup_{m\to\infty} a_m/m = \limsup_{j\to\infty} a_{m_j}/m_j\,,
    \qquad
    \liminf_{m\to\infty} a_m/m = \liminf_{j\to\infty} a_{m_j}/m_j\,.
  \end{equation*}\end{linenomath}
  This proves the second part of the lemma.
\end{proof}

\subsection{Local correlation entropy of $f^k$: The lower bound}

\begin{lemma}\label{L:c(fx)=c(x)}
  Let $(X,f)$ be a dynamical system, $m,h\in\NNN$, $x\in X$, and $\eps>0$.
  Then
  \begin{linenomath}\begin{equation*}
    \cccu^f_m(f^h(x),\eps) = \cccu^f_m(x,\eps)\,,
    \qquad
    \cccl^f_m(f^h(x),\eps) = \cccl^f_m(x,\eps)\,.
  \end{equation*}\end{linenomath}
\end{lemma}
\begin{proof}
  For every $n\in\NNN$ we easily have
  \begin{equation}\label{EQ:c(fx)=c(x)-1}
  \begin{split}
    &\left( \frac{n+h}{n} \right)^2   C^f_m(x,n+h,\eps) - \frac{2hn+h^2}{n^2}
    \ \le\
    C^f_m(f^h(x),n,\eps)
   \\
    &\ \le\
    \left( \frac{n+h}{n} \right)^2   C^f_m(x,n+h,\eps)  \,,
  \end{split}
  \end{equation}
  from which the lemma immediately follows.
\end{proof}

\begin{lemma}\label{L:c(fkx)=c(x)}
  Let $(X,f)$ be a dynamical system and $k,h\in\NNN$.
  Then for every $\eps>0$ there are $0<\gamma<\delta<\eps$ such that
  \begin{equation}\label{EQ:c(fkx)=c(x)}
  \begin{split}
    &\cccu^{f^k}_m(x,\gamma)
    \le
    \cccu^{f^k}_m(f^h(x),\delta)
    \le
    \cccu^{f^k}_m(x,\eps),
    \\
    &\cccl^{f^k}_m(x,\gamma)
    \le
    \cccl^{f^k}_m(f^h(x),\delta)
    \le
    \cccl^{f^k}_m(x,\eps)
  \end{split}
  \end{equation}
  for every $x\in X$ and $m\in\NNN$.
\end{lemma}
\begin{proof}
  Applying Lemma~\ref{L:c(fx)=c(x)} to $f^k$ allows us to assume that $h<k$.
  Since $f^{k-h}$ is uniformly continuous, there is $\delta\in(0,\eps)$ such that
  $\varrho(f^{k-h}(y),f^{k-h}(z))\le \eps$ whenever $\varrho(y,z)\le \delta$. This implies that
  $\varrho^{f^k}_m(f^{k-h}(y),f^{k-h}(z))\le \eps$ for every $y,z\in X$ with
  $\varrho^{f^k}_m(y,z)\le \delta$. Thus
  \begin{equation}\label{EQ:c(fkx)=c(x)-1}
    C^{f^k}_m(f^h(x),n,\delta)
    \le
    C^{f^k}_m(f^k(x),n,\eps)
    \qquad\text{for every } n.
  \end{equation}
  An analogous application of uniform continuity of $f^h$ gives that
  there is $\gamma\in (0,\delta)$ such that
  \begin{equation}\label{EQ:c(fkx)=c(x)-2}
    C^{f^k}_m(x,n,\gamma)
    \le
    C^{f^k}_m(f^h(x),n,\delta)
    \qquad\text{for every } n.
  \end{equation}
  Now (\ref{EQ:c(fkx)=c(x)-1}), (\ref{EQ:c(fkx)=c(x)-2}), and Lemma~\ref{L:c(fx)=c(x)} yield
  (\ref{EQ:c(fkx)=c(x)}).
\end{proof}

\begin{corollary}\label{C:entr(fhx)=entr(x)}
  Let $(X,f)$ be a dynamical system, $k,h\in\NNN$, and $x\in X$. Then
  \begin{linenomath}\begin{equation*}
    \entru(f^k,f^h(x)) = \entru(f^k,x),
    \qquad
    \entrl(f^k,f^h(x)) = \entrl(f^k,x).
  \end{equation*}\end{linenomath}
\end{corollary}

\begin{lemma}\label{L:entropy-fk-part1}
  Let $(X,f)$ be a dynamical system and $k\in\NNN$. Then for every $\eps>0$
  there is $\delta\in(0,\eps)$
  such that
  \begin{linenomath}\begin{equation*}
    \cccu^{f^k}_m(x,\delta)  \le \cccu^{f}_{km}(x,\eps),
    \qquad
    \cccl^{f^k}_m(x,\delta)  \le \cccl^{f}_{km}(x,\eps)
  \end{equation*}\end{linenomath}
  for every $m\in\NNN$ and $x\in X$.
\end{lemma}
\begin{proof}
  Since $X$ is compact and $f$ is continuous, there is $\delta\in(0,\eps)$ such that
  $\varrho(y,z)\le\delta$ implies $\varrho(f^h(y),f^h(z))\le \eps$ for every $h=0,\dots,k-1$. Hence
  \begin{linenomath}\begin{equation*}
    \varrho^f_{km}(y,z)\le \eps
    \qquad\text{for every }  y,z\in X
    \text{ with }
    \varrho^{f^k}_{m}(y,z)\le \delta.
  \end{equation*}\end{linenomath}
  This gives, for every $x\in X$ and $m,n\in\NNN$,
  \begin{linenomath}\begin{equation*}
    C^{f^k}_m(x,n,\delta)   \le   C^{f}_{km}(x,n,\eps).
  \end{equation*}\end{linenomath}
  Now the lemma immediately follows.
\end{proof}

\begin{corollary}\label{C:entr(fk)>=entr(f)}
  Let $(X,f)$ be a dynamical system, $k\in\NNN$, and $x\in X$. Then
  \begin{linenomath}\begin{equation*}
    \entru(f^k,x) \ge k \cdot\entru(f,x),
    \qquad
    \entrl(f^k,x) \ge k \cdot\entrl(f,x).
  \end{equation*}\end{linenomath}
\end{corollary}

\begin{proof}
  By Lemmas~\ref{L:c(fkx)=c(x)} and \ref{L:c(x)-via-subsequence},
  for every $\eps>0$ there is $\delta_\eps \in (0,\eps)$ such that
  \begin{linenomath}\begin{equation*}
  \begin{split}
    &\limsup_{m\to\infty} (-1/m) \log\cccl_m^{f^k}(x,\delta_\eps)
    \ge
    \limsup_{m\to\infty} (-1/m) \log\cccl_{km}^{f}(x,\eps)
    \\
    &\ =
    k\cdot \limsup_{m\to\infty} (-1/m) \log\cccl_{m}^{f}(x,\eps).
  \end{split}
  \end{equation*}\end{linenomath}
  Hence $\entru(f^k,x) \ge k \cdot\entru(f,x)$.
  The second inequality can be proved \break analogously.
\end{proof}

\subsection{Local correlation entropy of $f^k$: A combinatorial lemma}
\label{SS:f^k-combinatorial-lemma}
Fix a finite set $V$ consisting of $n$ points, and a partition
$\VVv=(V_0,V_1,\dots,V_{k-1})$ of it into $k\ge 2$ nonempty subsets.
Consider an undirected simple (not necessarily connected)
graph $G$ with the set of vertices $V$. The number of edges of $G$ 
is denoted by $m(G)$.
For $0\le a,b<k$, an edge $\{i,j\}$ of $G$ is called an $ab$-edge if $i\in V_a$
and $j\in V_b$, or vice versa.
We say that a graph $G$ is \emph{$\VVv$-admissible} if the following hold:

\begin{equation}\label{EQ:graph-Vadmissible}
\begin{split}
  &\text{If $\{i,j\},\{i',j\}$ are different edges of $G$ with $i,i'\in V_a$ and $j\in V_b$
  ($a\ne b$),}
\\
  &\text{then $\{i,i'\}$ is also an edge of $G$.}
\end{split}
\end{equation}
The number of all $ab$-edges of $G$ is denoted by $m_{ab}(G)$.
Put
\begin{equation}\label{EQ:def-kappa}
  \kappa(G) 
  = \sum_{a<b} m_{ab}(G) - (k-1)\sum_a m_{aa}(G)
  = m(G) - k\sum_a m_{aa}(G).
\end{equation}
Our aim is to find an upper bound for $\kappa(G)$ depending only on $n$ and $k$. To this end,
we say that a $\VVv$-admissible graph $G$ is \emph{$\VVv$-optimal} if $\kappa(G')\le \kappa(G)$
for every $\VVv$-admissible graph $G'$. Further, if $G$ is $\VVv$-optimal and
the number of edges of every $\VVv$-optimal
graph $G'$ is greater than or equal to that of $G$, we say that $G$ is
a \emph{minimal $\VVv$-optimal} graph.
The following lemma gives a characterization of minimal $\VVv$-optimal graphs.

\begin{lemma}\label{L:graph-Voptimal-1}
  Let $G$ be a graph with the set of vertices $V$.
  Then $G$ is a minimal $\VVv$-optimal graph if and only if
  the following two conditions hold for every $a\ne b$ from $\{0,\dots,k-1\}$:
  \begin{enumerate}
    \item[(a)] $m_{aa}(G)=0$ and $m_{ab}(G)=\min\{\card{V_a},\card{V_b}\}$;
    \item[(b)] no two $ab$-edges have a common vertex.
  \end{enumerate}
  Consequently,
  \begin{linenomath}\begin{equation*}
    \max \{\kappa(G): \ G\text{ is }\VVv\text{-admissible}\}
    \ = \
    \sum_{a<b} \min\{\card{V_a},\card{V_b}\}.
  \end{equation*}\end{linenomath}
\end{lemma}
\begin{proof}
We start by proving that
\begin{equation}\label{EQ:graph-Voptimal-1f}
  \kappa(G)
  \le 
  \sum_{a < b} \min\{\card{V_a},\card{V_b}\}
\end{equation}
for (every $V$, $\VVv$, and) every $\VVv$-admissible graph $G$;
clearly, it suffices to prove \eqref{EQ:graph-Voptimal-1f} for minimal
$\VVv$-optimal graphs $G$.
Assume first that $k=2$, i.e., $\VVv=\{V_0,V_1\}$. 
Fix a minimal $\VVv$-optimal graph $G$ and take any $a\ne b$ from $\{0,1\}$
(i.e., $a=0$ and $b=1$, or vice versa).
For $i\in V_a$ define
\begin{eqnarray*}
   A_{ib}  &=&  \{j\in V_b: \ \{i,j\}  \text{ is an edge of }G\},
   \\
   B_{ib}  &=& \{i'\in V_a: \ \{i',j\} \text{ is an edge of }G \text{ for some }j\in A_{ib}\}.
\end{eqnarray*}
Assume that $A_{ib}\ne\emptyset$.
Take the ($\VVv$-admissible) graph $\tilde G$ created from $G$ by removing
all $ab$-edges $\{i,j\}$ (with $j\in A_{ib}$) as well as
all $aa$-edges $\{i,i'\}$ (with $i'\in B_{ib}\setminus\{i\}$).
Then $\kappa(\tilde G) = \kappa(G) - \card{A_{ib}} + (2-1)(\card{B_{ib}}-1)$
since $i\in B_{ib}$.
Since $G$ is minimal,
we have that $\kappa(\tilde{G})<\kappa(G)$ and so
$\card{A_{ib}}\ge \card{B_{ib}}$.
If $A_{ib}=\emptyset$ then $B_{ib}=\emptyset$ by the definition of $B_{ib}$.
Thus, in both cases,
\begin{equation}\label{EQ:graph-Voptimal-1a}
    \card{A_{ib}}\ge \card{B_{ib}}.
\end{equation}

Assume again that $A_{ib}\ne\emptyset$.
Take any $j\in A_{ib}$ and define $A_{ja},B_{ja}$ analogously.
Then $B_{ja}\supseteq A_{ib}$ and $B_{ib}\supseteq A_{ja}$.
Inequality (\ref{EQ:graph-Voptimal-1a}),
applied also to $j$ and $a$, yields
$\card{A_{ib}} \le \card{B_{ja}} \le \card{A_{ja}} \le\card{B_{ib}}\le\card{A_{ib}}$.
Thus
\begin{linenomath}\begin{equation*}
  \card{A_{ib}} = \card{B_{ib}}
\end{equation*}\end{linenomath}
and, for every $j\in A_{ib}$,
\begin{linenomath}\begin{equation*}
  A_{ja} = B_{ib},
  \qquad
  B_{ja} = A_{ib}.
\end{equation*}\end{linenomath}
$\VVv$-admissibility of $G$ now gives that
$A_{ib}\cup B_{ib}$ is a clique of $G$ (that is, the induced subgraph is complete).
Since $G$ is minimal, this easily implies that $A_{ib}$ is a singleton.
(For if not, there is $l\ge 2$ such that we can write $A_{ib}=\{i_1=i,i_2,\dots,i_l\}$
and $B_{ib}=\{j_1=j,j_2,\dots,j_l\}$. Create a graph $\tilde G$ from $G$
by removing $l(l-1)$ edges $\{i_r,i_s\}$, $\{j_r,j_s\}$ ($r\ne s$)
and $l(l-1)$ edges $\{i_r,j_s\}$ ($r\ne s$). Then $\tilde G$ is $\VVv$-admissible,
$\kappa(\tilde{G})=\kappa(G)$, and $\tilde{G}$ has smaller number of edges
than $G$ --- a contradiction.)

We have proved that
\begin{equation}\label{EQ:graph-Voptimal-1b}
\begin{split}
\text{for}\text{ every }&a\ne b \text{ and every } i\in V_a,
\\
\text{there}\text{ is at }& \text{most one } ab\text{-edge from }i.
\end{split}
\end{equation}
Thus $m_{ab}(G)\le\card{V_a}$. Since analogously $m_{ab}(G) \le \card{V_b}$, we have 
\begin{linenomath}\begin{equation*}
  m_{ab}(G) \le \min \{\card{V_a}, \card{V_b}\}.
\end{equation*}\end{linenomath}
Moreover, by $\VVv$-optimality of $G$, $m_{aa}(G)=m_{bb}(G)=0$ (for if not,
by removing any $aa$-edge or any $bb$-edge we obtain a graph with larger $\kappa$
which is $\VVv$-admissible by \eqref{EQ:graph-Voptimal-1b}). 
Hence 
\begin{equation}\label{EQ:graph-Voptimal-1d}
\kappa(G) \le \min\{\card{V_0},\card{V_1}\}
\end{equation}
for every minimal $\VVv=\{V_0,V_1\}$-optimal graph $G$ and, consequently, for every
$\VVv=\{V_0,V_1\}$-admissible graph $G$.

\medskip
Now take any $k\ge 2$ and any partition $\VVv=\{V_0,V_1,\dots, V_{k-1}\}$ of $V$ 
into $k$ nonempty subsets. Let $G$ be any $\VVv$-admissible graph.
Fix any $a\ne b$ from $\{0,1,\dots,k-1\}$ and denote by $G_{ab}$ the subgraph of $G$
induced by the subset $V_a\cup V_b$ of the set $V$ of vertices of $G$.
Clearly, $G_{ab}$ is $\{V_a,V_b\}$-admissible and $m_{ab}(G_{ab})=m_{ab}(G)$, 
$m_{aa}(G_{ab})=m_{aa}(G)$, $m_{bb}(G_{ab})=m_{bb}(G)$. 
Hence, by \eqref{EQ:graph-Voptimal-1d} (applied to the set of vertices $V_a\cup V_b$,
partition $\{V_a,V_b\}$, and graph $G_{ab}$),
\begin{equation}\label{EQ:graph-Voptimal-1e}
  m_{ab}(G)-m_{aa}(G)-m_{bb}(G)
  =
  \kappa(G_{ab})
  \le 
  \min\{\card{V_a},\card{V_b}\}.
\end{equation}
Realize that $\kappa(G)$ can be written in the form
\begin{equation}\label{EQ:graph-Voptimal-1h}
  \kappa(G) = \frac 12 
    \sum_{a\ne b} \left(m_{ab}(G) - m_{aa}(G) - m_{bb}(G) \right) .
\end{equation}
This, together with \eqref{EQ:graph-Voptimal-1e} applied to every $a\ne b$, yields
\begin{linenomath}\begin{equation*}
  \kappa(G)
  \le 
  \frac 12 \sum_{a\ne b} \min\{\card{V_a},\card{V_b}\}
  = \sum_{a < b} \min\{\card{V_a},\card{V_b}\}
\end{equation*}\end{linenomath}
for every $\VVv$-admissible graph $G$. Thus, the proof of \eqref{EQ:graph-Voptimal-1f} is finished.

\bigskip
Now take any graph $H$ with the set of vertices $V$ which satisfies (a) and (b);
such a graph obviously exists.
By (b), $H$ is $\VVv$-admissible (indeed, the condition (\ref{EQ:graph-Vadmissible})
is trivially satisfied).
By (a), $\kappa(H)=\sum_{a<b} \min\{\card{V_a},\card{V_b}\}$.
Thus, by \eqref{EQ:graph-Voptimal-1f}, $H$ is $\VVv$-optimal.
For every graph $G$ (with the set of vertices $V$) having smaller number of edges 
than $H$ we have $\kappa(G)\le m(G)<m(H)=\kappa(H)$, hence $G$ is not 
$\VVv$-optimal. So $H$ is a minimal $\VVv$-optimal graph.

On the other hand, let $G$ be any minimal $\VVv$-optimal graph. 
By the previous part of the proof, $\kappa(G)=\sum_{a < b} \min\{\card{V_a},\card{V_b}\}$. 
This, together with \eqref{EQ:graph-Voptimal-1h} and \eqref{EQ:graph-Voptimal-1e}, give
\begin{equation}\label{EQ:graph-Voptimal-1g}
  m_{ab}(G)=m_{aa}(G)+m_{bb}(G) + \min\{\card{V_a},\card{V_b}\}
  \qquad\text{for every }a\ne b.
\end{equation}
Further, by \eqref{EQ:def-kappa},
\begin{linenomath}\begin{equation*}
  m(G)
  = 
  \sum_{a < b} \min\{\card{V_a},\card{V_b}\} + k\sum_a m_{aa}(G).
\end{equation*}\end{linenomath}
Take any graph $H$ from the previous paragraph and recall that $m(H)=\kappa(H)=\sum_{a < b} \min\{\card{V_a},\card{V_b}\}$. Minimality of $G$ gives that $m(G)\le m(H)$ and so
$m_{aa}(G)=0$ for every $a$. Now \eqref{EQ:graph-Voptimal-1g} yields that $G$ satisfies (a).
The fact that $G$ satisfies (b) easily follows from (a) and $\VVv$-admissibility of $G$. 
This finishes the proof of the lemma.
\end{proof}

\begin{lemma}\label{L:graph-Voptimal-2}
  Let $V$ be a finite set of cardinality $n$
  and $\VVv$ be a partition of it into $k\ge 2$ nonempty subsets.
  Then
  \begin{linenomath}\begin{equation*}
    \kappa(G)\le \frac{n(k-1)}{2}
  \end{equation*}\end{linenomath}
  for every $\VVv$-admissible graph $G$.
\end{lemma}
\begin{proof}
  We first prove that
  \begin{equation}\label{EQ:graph-Voptimal-2a}
    \sum_{h=0}^{k-1} hx_h \le
    \frac{k-1}{2}
  \end{equation}
  for every
  $x\in K=\big\{(x_0,\dots,x_{k-1})\in\RRR^k:\ \sum_h x_h=1,\ x_0\ge\dots\ge x_{k-1}\ge 0\big\}$.
  To this end, define a map $f:K\to \RRR$ by $f(x)=\sum_{h=0}^{k-1} hx_h$.
  Since $K$ is compact and $f$ is continuous, there is $\bar{x}\in K$ which maximizes $f$.
  Suppose that $\bar{x}_h>\bar{x}_{h+1}$ for some $h<k-1$. Define $x'\in \RRR^k$ by
  $x'_i=(\bar{x}_h + \bar x_{h+1})/2$ if $i\in\{h,h+1\}$, and $x'_i=\bar x_i$ otherwise.
  Then $x'\in K$ and $f(x')=f(\bar x) + (\bar x_h - \bar x_{h-1})/2 > f(\bar x)$, a contradiction.
  Thus $\bar x_h=1/k$ for every $h$ and (\ref{EQ:graph-Voptimal-2a}) follows

  Now we can prove Lemma~\ref{L:graph-Voptimal-2}.
  Put $n_h=\card{V_h}$ for $h=0,\dots,k-1$; we may assume that $n_0\ge n_1\ge\dots\ge n_{k-1}$.
  Let $G$ be a $\VVv$-admissible graph. By Lemma~\ref{L:graph-Voptimal-1}
  and (\ref{EQ:graph-Voptimal-2a}) with
  $x_h=n_h/n$,
  \begin{linenomath}\begin{equation*}
    \kappa(G)
    \le
    \sum_{h=0}^{k-1} h n_h
    =
    n\cdot \sum_{h=0}^{k-1} h x_h
    \le
    \frac{n(k-1)}{2}.
  \end{equation*}\end{linenomath}
\end{proof}

\subsection{Local correlation entropy of $f^k$: The upper bound}

\begin{lemma}\label{L:fk-graph-corollary}
  Let $(X,f)$ be a dynamical system, $k\ge 2$, $\eps>0$, $x\in X$, and $m,n\in\NNN$. Then
  \begin{linenomath}\begin{equation*}
    C^f_{km}(x,kn,\eps)
    \le
    \frac{1}{k} \, \sum_{h=0}^{k-1} C^{f^k}_{m}(f^h(x),n,2\eps).
  \end{equation*}\end{linenomath}
\end{lemma}
\begin{proof}
  Put $\bar n=kn$, $V=\{0,1,\dots,\bar{n}-1\}$ and, for $0\le a<k$,
  $V_a=\{i\in V:\ i\equiv a\ (\mmod\, k)\}$.
  Let $G$ be an undirected simple graph with the set of vertices $V$ and such that,
  for any $i\ne j$ from $V$, $\{i,j\}$
  is an edge of $G$ if and only if
  \begin{linenomath}\begin{equation*}
    \varrho^f_{km}(f^i(x),f^j(x)) \le
    \begin{cases}
      2\eps  &\text{if } i,j\in V_a \text{ for some } a;
    \\
      \eps  &\text{otherwise.}
    \end{cases}
  \end{equation*}\end{linenomath}
  Notice that the number $m(G)$ of edges of $G$ satisfies
  \begin{equation}\label{EQ:fk-graph-corollary-1}
    m(G) \ge
    \frac 12\left[
      \bar{n}^2 C^f_{km}(x,\bar{n},\eps) - \bar{n}
    \right].
  \end{equation}

  Further, $G$ is $\VVv$-admissible. In fact, fix any $a\ne b$, different $i,i'\in V_a$,
  and $j\in V_b$.
  If $\{i,j\}$, $\{i',j\}$ are $ab$-edges, then
  $\varrho^f_{km}(f^i(x),f^j(x)) \le \eps$ and $\varrho^f_{km}(f^{i'}(x),f^j(x)) \le \eps$.
  Hence, by the triangle inequality,
  $\varrho^f_{km}(f^i(x),f^{i'}(x)) \le 2\eps$ and so $\{i,i'\}$ is an edge of $G$.
  Lemma~\ref{L:graph-Voptimal-2} and \eqref{EQ:def-kappa} yield
  \begin{equation}\label{EQ:fk-graph-corollary-2}
    m(G)
    =
    \sum_a m_{aa}(G) + \sum_{a<b} m_{ab}(G)
    \le k \sum_a m_{aa}(G) + \frac{\bar{n}(k-1)}{2} \,.
  \end{equation}

  Since $\varrho_m^{f^k}\le \varrho_{km}^f$,
  for every $0\le a<k$ the definition of $G$ gives
  \begin{linenomath}\begin{equation*}
    n^2\cdot C^{f^k}_m(f^a(x), n, 2\eps)
    \ge
    2m_{aa}(G) + n.
  \end{equation*}\end{linenomath}
  This together with (\ref{EQ:fk-graph-corollary-1}) and (\ref{EQ:fk-graph-corollary-2}) yield
  \begin{linenomath}\begin{equation*}
    \frac 12\left[
      \bar{n}^2 C^f_{km}(x,\bar{n},\eps) - \bar{n}
    \right]
    \le
    m(G)
    \le
    \frac k2 \sum_a\left[ n^2\cdot C^{f^k}_m(f^a(x), n, 2\eps)   - n  \right]
    + \frac{\bar{n}(k-1)}{2} \,.
  \end{equation*}\end{linenomath}
  Now a simple computation gives the desired inequality.
\end{proof}

\begin{lemma}\label{L:c(fkx)=c(x)-v2}
  Let $(X,f)$ be a dynamical system and $0\le h<k$ be integers.
  Then for every $\eps>0$ there is $\eta(\eps)>0$ such that
  \begin{equation}\label{EQ:c(fkx)=c(x)-v2}
    \lim_{\eps\to 0}\eta(\eps)=0
    \qquad\text{and}\qquad
    C^{f^k}_m(f^h(x), n,\eps)
      \le
      C^{f^k}_{m}(x, n+1,\eta(\eps)) + \frac 3n
  \end{equation}
  for every $x\in X$ and $m,n\in\NNN$.
\end{lemma}
\begin{proof}
In the proof of Lemma~\ref{L:c(fkx)=c(x)} we have shown that for every $e>0$ there is
$d(e)\in (0,e)$
such that $C^{f^k}_m(f^h(x),n,d(e)) \le C^{f^k}_m(f^k(x),n,e)$; see (\ref{EQ:c(fkx)=c(x)-1}).
Fix a sequence $(e_i)_{i\ge 0}$ decreasing to zero and put $d_i=d(e_i)$; we may assume that
$d_i>d_{i+1}$ for every $i$. For every $\eps>0$ define
\begin{linenomath}\begin{equation*}
  \eta(\eps) =
  \begin{cases}
    e_i     &\text{if } \eps\in (d_{i+1},d_i] \text{ for some } i;
  \\
    \diam(X)   &\text{if } \eps > d_0.
  \end{cases}
\end{equation*}\end{linenomath}
Since $d_i\searrow 0$, $\eta(\eps)$ is defined for every $\eps>0$; further,
$\eps<\eta(\eps)$ for every $\eps\in (0, d_0]$.
Thus $C^{f^k}_m(f^h(x),n,\eps) \le C^{f^k}_m(f^k(x),n,\eta(\eps))$
for every $\eps>0$.
Combining this with (\ref{EQ:c(fx)=c(x)-1}), applied to $f'=f^k$, $h'=1$, and $\eps'=\eta(\eps)$, 
yields
\begin{linenomath}\begin{equation*}
  C^{f^k}_m(f^h(x),n,\eps)
  \le
  \left( \frac{n+1}{n} \right)^2   C^{f^k}_m(x,n+1,\eta(\eps))
  \le
  C^{f^k}_m(x,n+1,\eta(\eps)) + \frac 3n
  \,.
\end{equation*}\end{linenomath}
Since $\lim_\eps \eta(\eps)= 0$ is immediate by the choice of $\eta$, the lemma is proved.
\end{proof}

\begin{proof}[Proof of Theorem~\refThmA{}]
  We may assume that $k\ge 2$.
  Lemma~\ref{L:c(fkx)=c(x)-v2}, applied to every $h\in\{0,\dots,k-1\}$, gives that
  for every $\eps>0$ there is $\eta(\eps)>0$ such that $\lim_{\eps\to 0} \eta(\eps)=0$
  and
  \begin{linenomath}\begin{equation*}
    C^{f^k}_m(f^h(x), n,\eps)
      \le
      C^{f^k}_{m}(x, n+1,\eta(\eps)) + \frac 3n
  \end{equation*}\end{linenomath}
  for every $0\le h<k$, $x\in X$, and $m,n\in\NNN$.
  Now, by Lemma~\ref{L:fk-graph-corollary},
  \begin{linenomath}\begin{equation*}
  \begin{split}
    &C^f_{km}(x,kn,\eps)
   \le
    \frac{1}{k} \, \sum_{h=0}^{k-1} C^{f^k}_{m}(f^h(x),n,2\eps)
   \\
   &\ \le
    C^{f^k}_{m}(x,n+1,\eta(2\eps))
    +\frac{3}{n} \,.
  \end{split}
  \end{equation*}\end{linenomath}
  By taking the limit as $n$ approaches infinity, and using Lemma~\ref{L:c(x)-via-subsequence}
  we obtain
  \begin{linenomath}\begin{equation*}
    \cccu^{f}_{km}(x,\eps) \le \cccu^{f^k}_m(x,\eta(2\eps)),
    \qquad
    \cccl^{f}_{km}(x,\eps) \le \cccl^{f^k}_m(x,\eta(2\eps)).
  \end{equation*}\end{linenomath}
  Consequently, again using Lemma~\ref{L:c(x)-via-subsequence},
  \begin{linenomath}\begin{equation*}
    k\cdot \entru(f,x) \ge \entru(f^k,x),
    \qquad
    k\cdot \entrl(f,x) \ge \entrl(f^k,x).
  \end{equation*}\end{linenomath}
  Since the opposite inequalities were shown in Corollary~\ref{C:entr(fk)>=entr(f)},
  Theorem~\refThmA{} is proved.
\end{proof}

\section{Proof of Theorem~\refThmB}\label{S:proofB}

\begin{lemma}\label{L:c(x)>=1/r_m(eps)}
Let $(X,f)$ be a dynamical system, $x\in X$, $\eps>0$, and $m,n\in\NNN$. Then
\begin{linenomath}\begin{equation*}
  C^f_m(x,n,\eps) \ge \frac{1}{r_m(\eps/2, X)}.
\end{equation*}\end{linenomath}
\end{lemma}

\begin{proof}
The proof is pretty similar to that of Lemma~\ref{L:c(x)>=eta^m}; the
only difference is that instead of $(\eps/2)$-spanning sets we use $(m,\eps/2)$-spanning sets.
For completeness, the details follow.

Let $\{y_0,\dots,y_{p-1}\}$ be an $(m,\eps/2)$-spanning subset of
minimal cardinality $p=r_m(\eps/2,X)$.
Hence for every $i\ge 0$ and $x_i=f^i(x)$ there is $v_i$ with
$\varrho^f_m(x_i, y_{v_i})\le\eps/2$. For $0\le v<p$ put
\begin{linenomath}\begin{equation*}
  N_{v}=\{0\le i\le n-m:\ v_i=v\}
  \qquad\text{and}\qquad
  n_v=\card{N_v}.
\end{equation*}\end{linenomath}
Then, by the arithmetic-quadratic mean inequality,
\begin{linenomath}\begin{equation*}
  C^f_m(x,n,\eps)
    \ge   \frac{1}{n^2}   \cdot \sum_{v<p} n_v^2
    \ge   \frac{1}{n^2}   \cdot \frac{n^2}{p}
    = \frac{1}{r_m(\eps/2, X)} \,.
\end{equation*}\end{linenomath}
\end{proof}

\begin{proposition} \label{P:correl-and-topol-entropy-upper-bound}
  Let $(X,f)$ be a dynamical system and $x\in X$. Then
  \begin{linenomath}\begin{equation*}
    \entrl(f,x) \le \entru(f,x) \le h_\topol\left(f|_{\closure\orbit_f(x)}\right) \le h_\topol(f).
  \end{equation*}\end{linenomath}
\end{proposition}
The part corresponding to the lower local correlation entropy
was proved in \cite[p.~354]{takens1983invariants}.
The proof used the fact that if $x$
is a quasi-generic point \cite[(4.4)]{denker1976ergodic} of an invariant measure $\mu$,
then \cite[p.~355]{takens1983invariants}
\begin{linenomath}\begin{equation*}
  \entrl(f,x) \le \entrl(f,\mu) \le h_\mu(f).
\end{equation*}\end{linenomath}

\begin{proof}
  By Lemma~\ref{L:c(x)>=1/r_m(eps)} and Bowen's definition of topological entropy,
  \begin{linenomath}\begin{equation*}
  \begin{split}
    &\entru(f,x)
    =    \lim_{\eps\to 0} \limsup_{m\to\infty} (-1/m)\log \cccl^f_m(x,\eps)
    \\
    &\ \le  \lim_{\eps\to 0} \limsup_{m\to\infty} (1/m) \log r_m(\eps/2,X)
    = h_\topol(f).
  \end{split}
  \end{equation*}\end{linenomath}
  Applying this to $X'=\closure\orbit_f(x)$ and $f'=f|_{X'}$ yields the required inequality.
\end{proof}

\begin{remark}
  Proposition~\ref{P:correl-and-topol-entropy-upper-bound} is tightly connected with the fact that,
  for every $f$-invariant measure $\mu$,
  $\entru(f,\mu) \le h_\mu(f)$ (see Proposition~\ref{P:takens-corrEntropy-measureEntropy}).
  Thus, by (\ref{EQ:entr-strong-law}),
  \begin{linenomath}\begin{equation*}
    \entru(f,x) \le h_\mu(f)
    \qquad\text{for } \mu \text{-a.e.~} x\in X
  \end{equation*}\end{linenomath}
  provided $\mu$ is ergodic.
\end{remark}

Now we embark on the proof of the fact that, for dynamical systems on topological graphs,
local correlation entropies can be arbitrarily close to the topological entropy.

\begin{proposition}\label{P:extension-of-subshift}
  Let $(X,f)$ be a dynamical system having a subsystem $(Y,f)$, which is a topological extension
  of the full shift $(\Sigma_p,\sigma)$ for some $p\ge 2$.
  Then there is $\eps_0>0$ such that the following is true:
  For every $\alpha\in\Sigma_p$ there is $x_\alpha\in Y$ such that
  $x_\alpha\ne x_\beta$ whenever $\alpha\ne\beta$, and
  \begin{linenomath}\begin{equation*}
    C^f_m(x_\alpha,n,\eps) \le C^\sigma_m(\alpha,n,\tfrac12)
    \qquad
    \text{for every }\eps\in(0,\eps_0]
    \text{ and } m,n\in\NNN.
  \end{equation*}\end{linenomath}
  Consequently,
  \begin{linenomath}\begin{equation*}
    \entru(f,x_\alpha)\ge\entru(\sigma,\alpha)
    \qquad\text{and}\qquad
    \entrl(f,x_\alpha)\ge\entrl(\sigma,\alpha).
  \end{equation*}\end{linenomath}
\end{proposition}
\begin{proof}
  Let $h:(Y,f)\to (\Sigma_p,\sigma)$ be a factor map (that is, $h$ is a continuous surjection
  and $h\circ f = \sigma\circ h$).
  For every $j\in\AAa_p=\{0,\dots,p-1\}$ put $Y_j = h^{-1}([j])$ (recall that $[j]$ denotes
  the cylinder $\{\alpha\in\Sigma_p:\ \alpha_0=j\}$); this is a closed, hence compact set.
  Put $\eps_0=\frac12\min \{\dist(Y_i,Y_j):\ {i\ne j}\}$; since the sets $Y_j$ are pairwise disjoint
  and compact, we have
  $\eps_0>0$.

  Fix any $\alpha=\alpha_0\alpha_1\ldots\in\Sigma_p$ and take arbitrary
  $x=x_\alpha\in h^{-1}(\{\alpha\})$;
  clearly, $x_\alpha\ne x_\beta$ whenever $\alpha\ne\beta$.
  Realize that $f^i(x)\in Y_{\alpha_i}$ for every $i$.
  Hence, by the choice of $\eps_0$, $\varrho(f^i(x),f^j(x))\le\eps_0$ implies $\alpha_i=\alpha_j$.
  Thus also, for every $i$ and $j$,
  \begin{linenomath}\begin{equation*}
    \varrho^f_m(f^i(x),f^j(x))\le\eps_0
    \qquad\text{implies}\qquad
    \tilde\varrho^\sigma_m(\sigma^i(\alpha),\sigma^j(\alpha))\le \tfrac12
  \end{equation*}\end{linenomath}
  (where $\tilde\varrho$ denotes the metric on $\Sigma_p$, see \textsection\ref{SS:shifts};
  recall that $\tilde\varrho(\alpha,\beta)\le \frac12$ is
  equivalent to $\alpha_0=\beta_0$).
  Now $C^f_m(x,n,\eps)\le C^\sigma_m(\alpha,n,\frac12)$ for every $\eps\in(0,\eps_0]$
  and $m,n\in\NNN$,
  from which the first assertion immediately follows.

  The second assertion then follows by Lemma~\ref{L:subshifts-correl-entropy-local}.
  To see this, assume that $\eps_0\le 1$. For every $\eps\in(0,\eps_0]$ denote by $k_\eps$
  the unique nonnegative integer such that $\eps\in \big[2^{-k_\eps},2^{-(k_\eps-1)} \big)$.
  Then, by Lemma~\ref{L:subshifts-correl-entropy-local},
  $C_m^\sigma(\alpha,n,\frac12)
   =C_{m-k_\eps+1}^\sigma(\alpha,n,2^{-k_\eps})
   =C_{m-k_\eps+1}^\sigma(\alpha,n,\eps)$.
   So, by the first part of the lemma,
   \begin{linenomath}\begin{equation*}
   \begin{split}
     &\limsup_{m\to\infty} (-1/m) \log\cccl_m^f(x_\alpha,\eps)
     \ge
     \limsup_{m\to\infty} (-1/m) \log\cccl_{m-k_\eps+1}^\sigma(\alpha,\eps)
     \\
     &\ =
     \limsup_{m\to\infty} (-1/m) \log\cccl_{m}^\sigma(\alpha,\eps)
   \end{split}
   \end{equation*}\end{linenomath}
   and $\entru(f,x_\alpha)\ge\entru(\sigma,\alpha)$. Analogously for lower entropies.
\end{proof}

Recall that subsets $X_0,\dots,X_{p-1}$ of $X$
form a \emph{strict $p$-horseshoe} of a dynamical system $(X,f)$
if the sets $X_i$ are nonempty, closed, pairwise disjoint, and
$f(X_i)\supseteq \bigcup_j X_j$ for every $0\le i<p$.

\begin{lemma}\label{L:horseshoe=>subsystem}
  Let $(X,f)$ be a dynamical system containing a strict $p$-horseshoe $X_0,\dots,X_{p-1}$
  for some $p\ge 2$.
  Then $(X,f)$ has a subsystem $(Y,f)$ which is a topological extension of the full shift
  $(\Sigma_p,\sigma)$.
\end{lemma}
\begin{proof}
  This is standard.
  Since the sets $X_0,\dots,X_{p-1}$ form a strict $p$-horseshoe, in a
  usual way for every $k\ge 2$ we can construct disjoint nonempty
  compact subsets $X_a$ ($a\in\AAa_p^k$) such that
  \begin{linenomath}\begin{equation*}
   f(X_{a_0a_1\dots a_{k-1}})=X_{a_1a_2\dots a_{k-1}}
   \qquad\text{and}\qquad
   X_{a_0a_1\dots a_{k-1}} \subseteq X_{a_0a_1\dots a_{k-2}}
  \end{equation*}\end{linenomath}
  for every $a=a_0a_1\dots a_{k-1}\in\AAa_p^k$. For
  $\alpha=\alpha_0\alpha_1\ldots\in\Sigma_p$ put
  $X_\alpha=\bigcap_{k\ge 1} X_{\alpha_0\dots\alpha_{k-1}}$. Then
  $Y=\bigcup_\alpha X_\alpha$ is compact, $\sigma(Y)=Y$, and
  $(Y,f|_Y)$ is a topological extension of the full shift $(\Sigma_p,\sigma)$.
\end{proof}

Now we are ready to prove Theorem~\refThmB{}.
\begin{proof}[Proof of Theorem~\refThmB{}]
  By Proposition~\ref{P:correl-and-topol-entropy-upper-bound} it suffices to prove the second part
  of the theorem. We may assume that $h_\topol(f)>0$.
  Take arbitrary $0<h<h_\topol(f)$. By \cite{llibre1993horseshoes}
  there are integers $p,k$ with $(1/k)\log p\ge h$ such that $f^k$ has a strict $p$-horseshoe.
  By Corollary~\ref{C:subshifts-correl-entropy-bernoulli},
  Lemma~\ref{L:horseshoe=>subsystem}, and Proposition~\ref{P:extension-of-subshift},
  there is a Cantor set $X_h$ such that $\entrl(f^k,x)\ge \log p$ for every $x\in X_h$.
  Hence, by Theorem~\refThmA{}, $\entrl(f,x)=(1/k)\,\entrl(f^k,x)\ge (1/k) \log p \ge h$
  for every $x\in X_h$.
\end{proof}

\begin{remark}[Infimum of local correlation entropies]
For every continuous map $f\colon X\to X$ of a topological graph $X$
we always have
\begin{linenomath}\begin{equation*}
 \inf_{x\in X} \entrl(f,x) = \inf_{x\in X} \entru(f,x) = 0.
\end{equation*}\end{linenomath}
This follows from Proposition~\ref{P:correl-and-topol-entropy-upper-bound}
and from the fact that positive entropy maps of topological graphs
have (dense) periodic points.
\end{remark}

The following two examples show that it can happen that the
local correlation entropy at every point is strictly smaller than the
topological entropy of $f$ and that, in positive entropy systems on topological graphs,
the set of those $x$
with positive local correlation entropy can be negligible from the
measure-theoretic point of view.

\begin{example} Take $\lambda\in (0,\infty]$. For $n\in\NNN$ let $I_n=[1/(n+1),1/n]$  and let
$f_n:I_n\to I_n$ be such that it fixes the end points of $I_n$,
$h_\topol(f_n)<\lambda$ and $\sup_n h_\topol(f_n)=\lambda$. Define a map $f:I\to
I$ by
\begin{linenomath}\begin{equation*}
 f(0)=0,\qquad
 f(x)=f_n(x) \text{ if } x\in I_n, n\ge 1.
\end{equation*}\end{linenomath}
Then $f$ is continuous and $h_\topol(f)=\lambda$
(see e.g.~\cite[Theorem~11.2]{pesin1997dimension}).
On the other hand, for
every $x$ we have $\entru(f,x)<\lambda$. In fact, if $x=0$ then
$\entr(f,x)=0$ since $x$ is fixed, and if $x\in I_n$ then $\entru(f,x)\le
h_\topol(f_n)<\lambda$ by Proposition~\ref{P:correl-and-topol-entropy-upper-bound}.
\end{example}

\begin{example}
Let $f:[-1,1]\to [-1,1]$ be defined by $f(x)=1-\alpha x^2$, where
$\alpha\in (1,2)$ is such that $1-\alpha(1-\alpha)^2=0$. Then
almost every point $x$ is attracted by the $3$-periodic orbit of the
point $0$ \cite[p.~119]{collet1980iterated} and hence $\entr(f,x)=0$.
On the other hand, having a point with period $3$, the topological entropy of
$f$ is positive.
\end{example}

\section{Uniquely ergodic systems}\label{S:uniqueErg}
In this section we summarize facts on uniquely ergodic systems, which will be used
in Section~\ref{S:strictlyErgExample}.
Following \cite{grillenberger1973constructions},
we say that a set $A\subseteq\NNN_0$ is \emph{uniform Ces\`aro with density $\alpha\ge 0$}
if for every $\eps>0$ there is $n_0\in\NNN$ such that, for every $n\ge n_0$ and $j\in\NNN_0$,
\begin{equation}\label{EQ:unifCesaro-def}
  \left|
    \frac{1}{n} \cdot \card{A\cap [j,j+n)} \ - \ \alpha
  \right|
  \ < \ \eps.
\end{equation}
In such a case the density $\alpha$ of $A$ will be denoted by $d(A)$.
It is easy to check that $A\subseteq\NNN_0$ is uniform Ces\`aro with density $\alpha$
if and only if there is $l\in\NNN$ such that for every $\eps>0$ there is $n_0\in\NNN$ with
\begin{equation}\label{EQ:unifCesaro-def2}
  \left|
    \frac{1}{ln} \cdot \card{A\cap [lj,lj+ln)} \ - \ \alpha
  \right|
  \ < \ \eps
  \qquad\text{for every }n\ge n_0 \text{ and } j\in\NNN_0.
\end{equation}

Let $p\ge 2$. For words $u,v\in\AAa_p^*$ with $\len{u}\le\len{v}$ and an integer $l\ge 1$ put
\begin{linenomath}\begin{equation*}
  N_v^{(l)}(u) = \{i\in\NNN_0:\ v[il,il+\abs{u})=u\},
  \qquad
  \tau_v^{(l)}(u) = \frac{1}{\big\lfloor \len{v}/l\big\rfloor} \cdot \cardbig{N_v^{(l)}(u)};
\end{equation*}\end{linenomath}
so $\tau_v^{(l)}(u)\le 1$ is the frequency of occurrences of $u$ in $v$ at positions
which are multiples of $l$.
(Since $N_v^{(l)}(u)\subseteq \big[0,\lfloor(\len{v}-\len{u})/l\rfloor\big]$,
in the definition of $\tau_v^{(l)}(u)$ we should divide by $1+\lfloor (\len{v}-\len{u})/l\rfloor$;
the difference is, of course, asymptotically negligible.)
If $u\in\AAa_p^*$ and $x\in\Sigma_p$, define $N_x^{(l)}(u)$ analogously.
For abbreviation, we often write $N_x,\tau_x$ and $N_v,\tau_v$ instead of
$N_x^{(1)}, \tau_x^{(1)}$ and $N_v^{(1)}, \tau_v^{(1)}$.
Note that $N_x^{(l)}(u)$ is uniform Ces\`aro if and only if for every $u\in\AAa_p^*$ the limit
$\lim_{n} \tau_{x[jl, (j+n)l)}^{(l)}(u)$ exists uniformly in $j$ and does not depend on $j$;
in such a case we have
\begin{equation}\label{EQ:uniqErgod-density-Nxl}
  d\big(N_x^{(l)}(u)\big)
  =
  (1/l) \lim_{n\to\infty} \tau_{x[jl, (j+n)l)}^{(l)}(u)
  \qquad
  \text{for every } j.
\end{equation}

By \cite[Theorem~3.9]{hahn1967entropy}
we have the following.

\begin{lemma}[\cite{hahn1967entropy}]\label{L:uniqErgod-unifCesaro}
  Let $x\in\Sigma_p$ be almost periodic. Assume that $N_x(u)$ is uniform Ces\`aro
  for every $u\in\AAa_p^*$.
  Then the subshift $(\closure\orbit_\sigma(x),\sigma)$ is strictly ergodic. Moreover,
  \begin{linenomath}\begin{equation*}
    \mu([u])
    = d\left( N_x(u) \right)
    = \lim_{n\to\infty} \tau_{x[j, j+n)}(u)
    \qquad\text{for every } u\in\AAa_p^* \text{ and } j\in\NNN_0,
  \end{equation*}\end{linenomath}
  where $\mu$ is the unique invariant measure of $(\closure\orbit_\sigma(x),\sigma)$.
\end{lemma}

The following lemma gives a condition on $x$ implying
strict ergodicity of its orbit closure.

\begin{lemma}\label{L:strictlyErg-uniqueErg-equivCond}
  Let $x\in\Sigma_p$ be almost periodic and
  let $(l_j)_{j\ge 1}$ be an increasing sequence of positive integers
  with every $l_{j+1}$ being a multiple of $l_j$.
  Assume that, for every $j\ge 1$ and every $l_j$-word $v$, the set
  \begin{linenomath}\begin{equation*}
    N_x^{(l_j)}(v) = \{i\in\NNN_0:\ x[il_j,(i+1)l_j)=v\}
  \end{equation*}\end{linenomath}
  is uniform Ces\`aro.
  Then the subshift $(\closure\orbit_\sigma(x),\sigma)$ is strictly ergodic.
\end{lemma}
\begin{proof}
  The proof is inspired by that of \cite[Lemma~1.9]{grillenberger1973constructions}.
  Fix any nonempty word $u\in\AAa_p^*$; we want to prove that $N_x(u)$ is uniform Ces\`aro.
  Take $j$ such that $l=l_j>\len{u}$.
  Further, take arbitrary integers $1\le r<t$ and $0\le s$; for abbreviation, write
  $N_{s t}^{(\cdot)}$ and $\tau_{s t}^{(\cdot)}$
  instead of $N_{x[sl,(s+t)l)}^{(\cdot)}$ and $\tau_{x[sl,(s+t)l)}^{(\cdot)}$.

  We first prove that
  \begin{equation}\label{EQ:strictlyErg-uniqueErg-equivCond1}
    0
    \le
      \tau_{st}^{(1)}(u)
      - \sum_{v\in\AAa_p^{rl}} \tau_{s t}^{(l)}(v) \cdot \tau_v^{(1)}(u)
    <
    \frac{1}{r} + \frac{2r}{t}\,.
  \end{equation}
  To this end, for $i\in sl+N_{st}^{(1)}(u) \subseteq [sl, (s+t)l)$ put
  \begin{equation}\label{EQ:strictlyErg-uniqueErg-equivCond2}
  \begin{split}
    B_i
    &=
    \Big\{h\in\NNN_0:\ [i,i+\len{u})\ \subseteq\  [hl,(h+r)l)\ \subseteq\  [sl, (s+t)l)\Big\}
    \\
    &= [s,s+t-r]
      \cap \Big[
              \left\lceil ({i+\len{u}})/{l}  \right\rceil - r, \
              \left\lfloor {i}/{l}  \right\rfloor
           \Big]
  \end{split}
  \end{equation}
  and
  \begin{linenomath}\begin{equation*}
    b_{s t} = \sum_{i\,\in\, sl+N_{st}^{(1)}(u)} \card{B_i}.
  \end{equation*}\end{linenomath}
  That is, $b_{s t}$ is the number of pairs $(i,h)$, where $i-sl\in N_{st}^{(1)}(u)$
  and $h\in B_i$.
  But every such pair $(i,h)$ corresponds (in a one-to-one way) to a triple
  $(v,h',i')$, where $v\in \AAa_p^{rl}$, $h'\in N_{s t}^{(l)}(v)$, and $i'\in N_v^{(1)}(u)$;
  to see this, put $v=x[hl,(h+r)l)$, $h'=h$, and $i'=i-hl$.
  Thus
  \begin{equation}\label{EQ:strictlyErg-uniqueErg-equivCond3}
    b_{s t}
    = \sum_{v\in\AAa_p^{rl}} \cardbig{N_{s t}^{(l)}(v)} \cdot \cardbig{N_v^{(1)}(u)} \,.
  \end{equation}
  Further, by (\ref{EQ:strictlyErg-uniqueErg-equivCond2}),
  $0\le \card{B_i} \le r$ for every $i$
  and, provided $(s+r)l \le i \le (s+t-r)l$, $\card{B_i} \ge r-1$.
  This gives
  \begin{linenomath}\begin{equation*}
    r \cdot \cardbig{N_{st}^{(1)}(u)}
    \ \ge\
    b_{st}
    \ >\
    (r-1) \cdot \left(  \cardbig{N_{st}^{(1)}(u)} - 2rl  \right)
  \end{equation*}\end{linenomath}
  and so
  \begin{equation}\label{EQ:strictlyErg-uniqueErg-equivCond4}
    0
    \ \le\
    r \cdot \cardbig{N_{st}^{(1)}(u)}  -  b_{st}
    \ <\
    \cardbig{N_{st}^{(1)}(u)} + 2(r-1)rl
    \ <\
    (t+2r^2)l.
  \end{equation}
  Since $\cardbig{N_{st}^{(1)}}=tl\tau_{st}^{(1)}$, $\cardbig{N_{st}^{(l)}}=t\tau_{st}^{(l)}$,
  and $\cardbig{N_{v}^{(1)}}=rl\tau_{v}^{(1)}$ for $v\in\AAa_p^{rl}$,
  dividing (\ref{EQ:strictlyErg-uniqueErg-equivCond4})
  by $trl$ and using (\ref{EQ:strictlyErg-uniqueErg-equivCond3})
  gives (\ref{EQ:strictlyErg-uniqueErg-equivCond1}).

  Now take any $\eps>0$. Let $j'\ge j$ be such that $l_{j'}/l > 1/\eps$; put $r=l_{j'}/l$ and
  $\eps'=\eps/\cardbig{\AAa_p^{rl}}$.
  By the assumption, for every word $v\in \AAa_p^{rl}$ the set $N_x^{(rl)}(v)$ is uniform Ces\`aro;
  put $d_v=l\cdot d\big(N_x^{(rl)}(v)\big)$.
  Thus, by (\ref{EQ:unifCesaro-def}),  we can find $j''>j'$ such that
  $\abs{\tau_{st}^{(l)}(v) - d_v} < \eps'$ for every $t\ge l_{j''}/l$ and every $v\in\AAa_p^{rl}$.
  We may assume that $j''$ is so large that $(2r/t)<\eps$.
  Put $d=\sum_{v\in\AAa_p^{rl}} d_v\tau_v^{(1)}(u)$.
  Then
  \begin{linenomath}\begin{equation*}
    \absbig{\sum_{v\in\AAa_p^{rl}}
       \left(
          \tau_{st}^{(l)}(v) \tau_v^{(1)}(u) -  d_v\tau_v^{(1)}(u)
       \right)}
    < \eps' \sum_{v\in\AAa_p^{rl}} \tau_v^{(1)}(u)
    \le \eps,
  \end{equation*}\end{linenomath}
  and (\ref{EQ:strictlyErg-uniqueErg-equivCond1}) gives
  \begin{linenomath}\begin{equation*}
    \absbig{\tau_{st}^{(1)}(u) - d}
    < 3\eps.
  \end{equation*}\end{linenomath}
  This is true for every sufficiently large $t$ and so, by (\ref{EQ:unifCesaro-def2}),
  the set $N_x^{(1)}(u)=\{i: x[i,i+\len{u})=u\}$
  is uniform Ces\`aro with density $d(A)=d$.
  Since $u$ was arbitrary,
  Lemma~\ref{L:uniqErgod-unifCesaro} yields strict ergodicity of
  the subshift $(\closure\orbit_\sigma(x),\sigma)$.
\end{proof}

\section{Proof of Theorem~\refThmC}\label{S:strictlyErgExample}

In this section we show that Theorem~\refThmB{} cannot be generalized to arbitrary
dynamical system. We construct a strictly ergodic system for which
local correlation entropy of every point is zero,
but the topological entropy is positive. The construction is a modification of that from
\cite[pp.~327--329]{grillenberger1973constructions}.

Fix an integer $p\ge 3$ and
take the alphabet $\AAa=\AAa_p=\{0,\dots,p-1\}$. Recall that
$\AAa^*=\bigcup_{m\ge 0} \AAa^m$ denotes
the set of all words over $\AAa$.
If $w,v$ are words, their concatenation is denoted by $wv$. Further,
for a word $w$ and a positive integer $n$, the concatenation $ww\dots w$ ($n$-times)
is denoted by $w^n$.

For $n\ge 1$ denote by $\PPp_{n}$ the set of all permutations
$\pi$ of $\{1,\dots,n\}$.
Write $\PPp_{n} = \{\pi_1^{(n)},\dots,\pi_{n!}^{(n)} \}$,
where $\pi_1^{(n)}$ denotes the identity.
For words $w_1,\dots,w_n\in\AAa^*$ and
$\pi\in\PPp_{n}$ define
\begin{linenomath}\begin{equation*}
  \pi(w_1, w_2,\dots, w_n) = w_{\pi(1)} w_{\pi(2)} \dots w_{\pi(n)}   \ \in\ \AAa^*.
\end{equation*}\end{linenomath}

Let $M=\{w_1<w_2<\dots<w_n\}$ be an ordered set of words over $\AAa$
(the order of $M$ need not be lexicographical) such that
the lengths $\len{w_i}$ are the same; denote their common value by $l(M)$.
For $r\ge 0$ let $M^{(r)}$ be the ordered set
\begin{equation}\label{EQ:Mr-order}
  M^{(r)}
  = \{ w_1^{(r)} < \dots < w_{n!} ^{(r)} \},
  \qquad\text{where}\quad
  w_j^{(r)} = w_1^r \,\pi_j^{n}(w_1, w_2, \dots, w_n).
\end{equation}
Note that the words $w_j^{(r)}$ are pairwise distinct, the length of every $w_j^{(r)}$ is
$l(M^{(r)}) = \big(\card{M}+r\big)\cdot l(M)$,
and the cardinality of $M^{(r)}$ is
$\cardbig{M^{(r)}} = \big(\card{M}\big)!$.
Further, $w_1^{(r)}$ starts with $(r+1)$ copies of $w_1$.

Let $M_1=\{0<1<\dots<p-1\}$. Then $l(M_1)=1$ and $\card{M_1}=p$.
Put $r_1=0$. If we have defined $M_j$ and $r_j$ for $j\ge 1$, define $M_{j+1}$ and $r_{j+1}$ by
\begin{equation}\label{EQ:Mj-def}
  M_{j+1} = M_j^{(r_j)}
  \qquad\text{and}\qquad
  r_{j+1} = \big\lceil \card{M_{j+1}}/l(M_{j+1}) \big\rceil \,.
\end{equation}
For every $j$ put
\begin{equation}\label{EQ:Mj-notation}
  m_j=\card{M_j},\qquad
  l_j=l(M_j),\qquad
  \lambda_j = (1 / l_j) \log m_j;
\end{equation}
let $\bar{w}_j$ denote the first (according to the order of $M_j$) word of $M_j$.
Note that for $j=1,2$ we have
\begin{equation}\label{EQ:m2-j2}
  l_1=1,\ m_1 = p,\ r_1=0,
  \qquad
  l_2=p,\ m_2=p\,!,\ r_2=(p-1)!\,.
\end{equation}
Further, for every $j\ge 1$,
\begin{equation}\label{EQ:mj-lj}
  m_{j+1} = {m_j!}
  \qquad\text{and}\qquad
  l_{j+1} = (m_j+r_j) l_j.
\end{equation}

Let $x\in\Sigma_p=\AAa^{\NNN_0}$ be the unique sequence such that
\begin{equation}\label{EQ:strictlyErg-x-def}
  x[0,l_j)=\bar{w}_j;
\end{equation}
such $x$ exists since $\bar{w}_{j+1}$ starts with ($r_j+1$ copies of) $\bar{w}_j$;
$x$ is unique since
$l_j=\len{\bar{w}_j}\nearrow\infty$ by Lemma~\ref{L:strictlyErg-Mj-props}(c) below.
Put $X=\closure{\orbit}_\sigma(x)$.

The proof of Theorem~\refThmC{} goes as follows.
First, in Lemmas~\ref{L:strictlyErg-minimality} and \ref{L:strictlyErg-unique_erg}
we show that the system $(X,\sigma)$ is strictly ergodic,
which will prove (a) of the theorem.
The fact that the topological entropy is positive is given in
Lemma~\ref{L:strictlyErg-top_entropy}.
Finally,
correlation entropies of the system are described in
Lemmas~\ref{L:strictlyErg-corrEntropy_mu}
and \ref{L:strictlyErg-local_corr_entropy}.
We start by summarizing some of the properties of the constructed sets $M_j$.

\begin{lemma}\label{L:strictlyErg-Mj-props}
  The following hold:
  \begin{enumerate}
    \item[(a)] $m_j/l_j$ is an even integer  provided $j\ge 2$, and so $r_j=m_j/l_j$
      (that is, ceiling in (\ref{EQ:Mj-def}) is unnecessary);
    \item[(b)] $r_{j}>p$ provided $j\ge 3$;
    \item[(c)] $\lim_j m_j=\lim_j l_j=\lim_j r_j=\infty$;
    \item[(d)] $l_{j+1} > p l_j^2$ provided $j\ge 3$;
    \item[(e)] $\sum_{j\ge 4} (1/l_j) < 1/(p l_3^2-1)$ and $l_3=(p+1)!$.
  \end{enumerate}
\end{lemma}

\begin{proof}
  (a)--(c)
  Immediately from the construction we have
  \begin{equation}\label{EQ:strictlyErg-mj_lj-increasing}
    1=l_1, \ p = l_2 < l_3 < \dots,
    \qquad
    3\le p=m_1 < m_2 < m_3 < \dots
  \end{equation}
  and
  \begin{equation}\label{EQ:strictlyErg-lj1<=2mjlj}
    l_{j+1}\le 2m_jl_j
    \qquad
    \text{for every }
    j\ge 1
  \end{equation}
  (the last inequality follows from $r_{j}\le m_{j}$, see (\ref{EQ:Mj-def})).
  Further, we claim that
  \begin{equation}\label{EQ:strictlyErg-lj<=mj-p}
    l_j\le m_j-p
    \qquad
    \text{for every }
    j\ge 2.
  \end{equation}
  Indeed, this is true for $j=2$ since $p\le p\,!-p$ (recall that $p\ge 3$).
  Assume that (\ref{EQ:strictlyErg-lj<=mj-p}) is true for some $j\ge 2$.
  By (\ref{EQ:strictlyErg-mj_lj-increasing}) , $2<l_j\le m_j-p<m_j-1$.
  Thus $m_{j+1}=m_j! > m_j(m_j-1)(m_j-p) 2 \ge (m_j-1)\cdot (2m_j l_j) > (p-1)l_{j+1}$. Now
  $m_{j+1}-l_{j+1} > (p-2) l_{j+1} > p$ and (\ref{EQ:strictlyErg-lj<=mj-p})
  is true also for $(j+1)$.

  By (\ref{EQ:m2-j2}), $m_1/l_1=p$ and $m_2/l_2=(p-1)!$ are integers.
  Assume now that $m_j/l_j$ is an integer for some $j\ge 2$.
  Then, by (\ref{EQ:mj-lj}) and the fact that $r_j=m_j/l_j$,
  \begin{linenomath}\begin{equation*}
    \frac{m_{j+1}}{l_{j+1}}
    = \frac{m_j!}{(m_j+r_j)l_j}
    = \frac{(m_j-1)!}{l_j+1}\,.
  \end{equation*}\end{linenomath}
  By (\ref{EQ:strictlyErg-lj<=mj-p}), this is an even integer, which is greater than $(m_j-2)!$.
  Thus (a) is proved and, since $m_j\ge m_2 = p\,!$ and $(m_2-2)! \ge m_2-2=p\,!-2\ge 2p-2>p$,
  also (b) is proved.
  Further, $\lim m_j=\lim l_j=\infty$ since these sequences are strictly monotone by
  (\ref{EQ:strictlyErg-mj_lj-increasing}), and $\lim r_j=\infty$ since, as we have just proved,
  $r_{j+1}> (m_j-2)!$ for $j\ge 2$. Thus we have (c).

  (d) By (\ref{EQ:mj-lj}), (a), and (b),
  $l_{j+1}=(m_j+r_j)l_j \ge m_j l_j = r_j l_j^2 > p l_j^2$ for $j\ge 3$.

  (e)
  A simple induction using (d) gives
  $l_{j+3} > p^{2^{j}-1} l_3^{2^j}\ge (p l_3^2)^j$ for every $j\ge 1$.
  Hence $\sum_{j\ge 4} (1/l_j) < 1/(pl_3^2-1)$. Since $l_3=(p\,! + (p-1)!) p=(p+1)!$ by
  (\ref{EQ:mj-lj}) and (\ref{EQ:m2-j2}), (e) is proved.
\end{proof}

\subsection{Strict ergodicity}
\begin{lemma}\label{L:strictlyErg-minimality}
  The subshift $(X,\sigma)$ is minimal.
\end{lemma}

\begin{proof}
  Take any word $u$ which occurs in $x$. Then there is $j$ such that
  $u$ occurs in $\bar{w}_j$. By the construction,
  $u$ occurs in every word from $M_{j+1}$. Since $x$ is a concatenation
  of words from $M_{j+1}$, we have that $x$ is almost periodic and $(X,\sigma)$ is minimal.
\end{proof}

\begin{lemma}\label{L:strictlyErg-Aw-unifCesaro}
  For every integer $j\ge 1$ and every word $v\in M_j$, the set
  \begin{linenomath}\begin{equation*}
    N_x^{(l_j)}(v)
    \ =\  \{i\in\NNN_0:\ x[i l_j,(i+1) l_j) = v\}
  \end{equation*}\end{linenomath}
  is uniform Ces\`aro with density
  \begin{linenomath}\begin{equation*}
    d\left(N_x^{(l_j)}(v) \right) = (1/l_{j+1})\cdot
    \begin{cases}
      (r_j+1)  &\text{if } w=\bar{w}_j;
    \\
      1  &\text{otherwise}.
    \end{cases}
  \end{equation*}\end{linenomath}
\end{lemma}

\begin{proof}
  Fix $j\ge 1$ and $v\in M_j$.
  Take arbitrary $u\in M_{j+1}$ and
  write $u={u}_0{u}_1 \dots {u}_{h-1}$, where $h=l_{j+1}/l_j$ and ${u}_i\in M_j$.
  Put $\bar\tau_u(v) = (1/h)\cdot\cardbig{\{i:\ {u}_i=v\}}=\tau_u^{(l_j)}(v)$.
  Then, by the construction,
  \begin{linenomath}\begin{equation*}
    \bar\tau_u(v) = (1/h)\cdot
    \begin{cases}
      (r_j+1)  &\text{if } v=\bar{w}_j;
    \\
      1  &\text{otherwise}
    \end{cases}
  \end{equation*}\end{linenomath}
  does not depend on $u$.
  Hence the cardinality of $N_x^{(l_j)}(v)\cap[il_{j+1}, (i+1)l_{j+1})$
  does not depend on $i$, and is equal to $(r_j+1)$ if $v=\bar{w}_j$
  and to $1$ otherwise. By (\ref{EQ:unifCesaro-def2}) and (\ref{EQ:uniqErgod-density-Nxl})
  the lemma follows.
\end{proof}

\begin{lemma}\label{L:strictlyErg-unique_erg}
  The subshift $(X,\sigma)$ is strictly ergodic.
\end{lemma}

\begin{proof}
  This immediately follows from Lemma~\ref{L:strictlyErg-uniqueErg-equivCond}.
  In fact, take any integer $j\ge 1$ and any $l_j$-word $v$.
  If $v\in M_j$, then $N_x^{(l_j)}(v)$ is uniform Ces\`aro
  by Lemma~\ref{L:strictlyErg-Aw-unifCesaro}.
  Otherwise $N_x^{(l_j)}(v)$ is empty, so again it is uniform Ces\`aro.
  Thus $(X,\sigma)$ is strictly
  ergodic by Lemmas~\ref{L:strictlyErg-uniqueErg-equivCond} and \ref{L:strictlyErg-minimality}.
\end{proof}

\subsection{Positive topological entropy}
Here we prove that the constructed subshift $(X,\sigma)$ has positive topological entropy.
We start with the crucial fact concerning the sequence $(\lambda_j)_j$, the
proof of which is a modification of that from \cite[Lemma~2.1]{grillenberger1973constructions}.

\begin{lemma}\label{L:strictlyErg-lambdaj_props}
  Let $p\ge 3$. Then the sequence $(\lambda_j)_{j\ge 1}$ is decreasing and
  the limit of it is positive.
\end{lemma}

\begin{proof}
  By (\ref{EQ:mj-lj}), $m_{j+1}=m_j!<m_j^{m_j}$. Hence, using that $r_j\ge 0$,
  \begin{linenomath}\begin{equation*}
    \lambda_{j+1}
    = \frac{\log m_{j+1}}{l_{j+1}}
    < \frac{m_j \log m_j}{(m_j+r_j) l_j}
    \le  \frac{\log m_j}{l_j}
    = \lambda_j.
  \end{equation*}\end{linenomath}
  Thus $(\lambda_j)_j$ is decreasing. Put $\lambda=\lim_j\lambda_j$; we are going to show that
  $\lambda>0$.
  To this end, recall Stirling's formula \cite{robbins1955remark}
  \begin{linenomath}\begin{equation*}
    n! = \sqrt{2\pi n} \left( \frac{n}{e}\right)^n  \cdot e^{\delta_n},
    \qquad\text{where}\quad
    \frac{1}{12n+1} < \delta_n < \frac{1}{12n}
    \qquad
    (n\ge 0).
  \end{equation*}\end{linenomath}
  For $n\ge 1$ this yields (using that $\delta_n>0$)
  \begin{equation}\label{EQ:Stirling2}
    \log {n!}
    > n(\log n - 1).
  \end{equation}
  So, by (\ref{EQ:mj-lj}),
  $l_{j+1}\lambda_{j+1}= \log m_{j+1} \ > \ m_j (\log m_j - 1)$.
  Dividing by $l_{j+1}=(m_j+r_j)l_j$, using (\ref{EQ:Mj-notation})
  and the facts that $\log m_j>0$ (recall that $\log$ means the natural logarithm)
  and $r_j=m_j/l_j>0$ for $j\ge 2$, we have
  \begin{linenomath}\begin{equation*}
    \lambda_{j+1}
    \ > \
    \frac{m_j (\log m_j -1)}{(m_j+r_j)l_j}
    \ > \
    \frac{\log m_j}{l_j}
    - \frac{r_j\log m_j}{m_j l_j}
    - \frac{1}{l_j}
  \end{equation*}\end{linenomath}
  for $j\ge 2$. Hence
  \begin{equation}\label{EQ:lambda-proof0}
    \lambda_{j+1}
    \ > \
    \lambda_j
    - \frac{1+\lambda_j}{l_j}
    \qquad
    \text{for every } j\ge 1
  \end{equation}
  (the inequality is trivial for $j=1$ since $l_1=1$ and $\lambda_2>0$).
  By monotonicity of $(\lambda_j)_j$, (\ref{EQ:lambda-proof0}) implies that, for $j\ge 1$,
  \begin{equation}\label{EQ:lambda-proof2}
    \lambda
    \ge
    \lambda_j - (\lambda_j+1) S_j,
    \qquad\text{where}\quad
    S_j=\sum_{k\ge j} \frac{1}{l_k}\,.
  \end{equation}
  We want to prove that $\lambda_j > S_j/(1-S_j)$ for some $j$;
  this fact together with (\ref{EQ:lambda-proof2}) will imply $\lambda >0$.

  Assume first that $p\ge 4$. We claim that
  \begin{equation}\label{EQ:lambda-proof-S2}
    S_2
    <   \frac{1}{p}    + \frac{1}{p(p^2-1)}
    \qquad\text{and}\qquad
    \frac{S_2}{1-S_2} < \frac{1}{p-2}
    \,.
  \end{equation}
  Indeed, by Lemma~\ref{L:strictlyErg-Mj-props}(e) and (\ref{EQ:m2-j2}),
  $S_2<1/p + 1/l_3 + 1/(pl_3^2-1)$ and $l_3=(p+1)!$.
  Since $p\ge 4$, $p l_3^2-1 > l_3\ge 2(p-1)p(p+1)$, we have the first inequality from
  (\ref{EQ:lambda-proof-S2}). The second one follows from the first and the facts that the map
  $x\mapsto x/(1-x)$ is increasing on $(-\infty,1)$, and that $1/p + 1/[p(p^2-1)] < 1$.

  By (\ref{EQ:m2-j2}) and (\ref{EQ:Stirling2}), $\lambda_2 > \log p-1$.
  Since $\log p > \frac32$ for $p\ge 5$,
  $\lambda_2 > \frac12 >S_2/(1-S_2)$ by (\ref{EQ:lambda-proof-S2}).
  For $p=4$ we have $\lambda_2=\log(24)/4$. Since $\log(24)>2$, we again have
  $\lambda_2 > \frac12 >S_2/(1-S_2)$. Thus, by (\ref{EQ:lambda-proof2}), $\lambda>0$
  for every $p\ge 4$.

  It remains to describe the case $p=3$.
  By (\ref{EQ:m2-j2}) and (\ref{EQ:mj-lj}), $m_3=720$, $l_3=24$,  and $\lambda_3=\log(720)/24$.
  Since $\log(720)>6$ (use e.g.~that $e<2.8$),
  we have $\lambda_3>\frac14$. On the other hand, by Lemma~\ref{L:strictlyErg-Mj-props}(e),
  $S_3<1/l_3 + 1/(p(l_3^2-1))<2/l_3=\frac{1}{12}$ and
  $S_3/(1-S_3)<\frac{1}{11}$. Thus, for $p=3$, $\lambda_3>S_3/(1-S_3)$ and $\lambda>0$ by
  (\ref{EQ:lambda-proof2}).
\end{proof}

\begin{lemma}\label{L:strictlyErg-top_entropy}
  Let $p\ge 3$. Then the topological entropy of the subshift $(X,\sigma)$ is
  \begin{linenomath}\begin{equation*}
    h_{\topol{}}(\sigma)=\lim_{j\to\infty} \lambda_j >0.
  \end{equation*}\end{linenomath}
\end{lemma}

\begin{proof}
  Put $\lambda=\lim_{j\to\infty} \lambda_j$; by Lemma~\ref{L:strictlyErg-lambdaj_props}
  the limit exists and is positive. We prove that $h_{\topol{}}(\sigma)=\lambda$;
  recall that $h_{\topol{}}(\sigma)=\lim_{n} (1/n)\log\theta_n$, where
  $\theta_n$ is the number of $n$-words in $x$
  (see e.g.~\cite[Theorem~7.13]{walters2000introduction}).
  The inequality $h_{\topol{}}(\sigma)\ge \lambda$ is trivial, since the number
  of $l_j$-words in $x$
  is greater than or equal to $m_j$. To prove the reverse inequality, take any
  $l_{j+1}$-word $v$ in $x$.
  Since $l_{j+1}=(m_j+r_j)l_j$, there are words $u_1,\dots,u_{m_j+r_j+1}$ from $M_j$
  and an integer $i\in[0,l_j)$ such that $v=(u_1\dots u_{m_j+r_j+1})[i,i+\len{v})$.
  From this fact it immediately follows that $\theta_{l_{j+1}} \le l_j m_j^{m_j+r_j+1}$. Thus
  \begin{linenomath}\begin{equation*}
    h_{\topol{}}(\sigma)
    =\lim_{j\to\infty} \frac{\log \theta_{l_{j+1}}} {l_{j+1}}
    \le \lim_{j\to\infty} \frac{\log l_j + (m_j+r_j+1)\log m_j} {(m_j+r_j)l_j}
    = \lambda.
  \end{equation*}\end{linenomath}
\end{proof}

\begin{remark}
  Since the beginning of this section we excluded the case $p=2$. Nevertheless,
  the construction can be carried over also for such $p$. The obtained subshift
  will be strictly ergodic
  (by the same reasoning as in
  Lemmas~\ref{L:strictlyErg-minimality}--\ref{L:strictlyErg-unique_erg}).
  However, the topological entropy will be zero. In fact, by (\ref{EQ:mj-lj}),
  \begin{linenomath}\begin{equation*}
    l_j=2\cdot 3^{j-2},\quad
    m_j=2,\quad
    r_j=1
    \qquad
    \text{for every }
    j\ge 2.
  \end{equation*}\end{linenomath}
  Hence, for $j\ge 2$, the number $\theta_{l_{j+1}}$ of $l_{j+1}$ words is less than or equal to
  $l_j m_j^{m_j+r_j+1} =16 l_j$ and $h_\topol{}(\sigma)=\lim_j (1/l_{j+1})\log \theta_{l_{j+1}}=0$.
\end{remark}

\subsection{Zero correlation entropy}
\begin{lemma}\label{L:strictlyErg-corrEntropy_mu}
  The correlation entropy $\entr(\sigma,\mu)$ of the unique invariant measure $\mu$
  of $(X,\sigma)$ is zero.
\end{lemma}
\begin{proof}
  Recall that, by Lemma~\ref{L:uniqErgod-unifCesaro} and the choice of $x$,
  \begin{equation}\label{EQ:strictlyErg-frequency-and-mu}
    \mu\big([v]\big) = \lim_{t\to\infty} \tau_{\bar{w}_t}(v)
    \qquad
    \text{for every } v\in\AAa^*.
  \end{equation}
  We start the proof by showing that
  \begin{equation}\label{EQ:strictlyErg-corrEntropy_mu-proof1}
    \mu\big([\bar{w}_j^k]\big)
    \ge \frac{r_j-k+1}{2m_jl_j}
    \qquad\text{for every }j\ge 1
    \text{ and } 1\le k\le r_j.
  \end{equation}
  To this end, fix any $j\ge 1$ and $1\le k\le r_j$. By the construction,
  every word $u$ from $M_{j+1}$ begins with $r_j$ copies of $\bar{w}_j$.
  Hence $u\big[il_j, (i+k)l_j\big)=\bar{w}_j^k$ for every $0\le i\le (r_j-k)$.
  By (\ref{EQ:mj-lj}) and Lemma~\ref{L:strictlyErg-Mj-props}(a),
  \begin{linenomath}\begin{equation*}
    \tau_u(\bar{w}_j^k)
    \ge \frac{r_j-k+1}{l_{j+1}}
    \ge \frac{r_j-k+1}{2m_j l_j}
    \qquad\text{for every } u\in M_{j+1}
    \,.
  \end{equation*}\end{linenomath}
  Now take any $t>j$. Since $\bar{w}_t$ is a concatenation of words from $M_{j+1}$,
  we have $\tau_{\bar{w}_t}\big(\bar{w}_j^k\big)\ge (r_j-k+1)/(2m_jl_j)$ for every $u\in M_t$.
  By (\ref{EQ:strictlyErg-frequency-and-mu}), this yields
  (\ref{EQ:strictlyErg-corrEntropy_mu-proof1}).

  Recall the definition (\ref{EQ:subshifts-muk-def}) of $\tilde\mu$.
  Take any $n\in\NNN$, put $w=x[0,n)$, and find $j$ such that $l_j\le n<l_{j+1}$; we may assume
  that $j\ge 2$.
  Assume first that $n<(r_j/2)l_j$; note that, by Lemma~\ref{L:strictlyErg-Mj-props}(a),
  $r_j/2$ is an integer.
  In this case
  $\mu\big([w]\big) \ge \mu\big(\big[\bar{w}_j^{r_j/2}\big]\big) \ge (r_j/2)/(2m_jl_j)
  = 1/(4l_j^2)$
  by (\ref{EQ:strictlyErg-corrEntropy_mu-proof1}) and Lemma~\ref{L:strictlyErg-Mj-props}(a).
  Thus
  \begin{equation}\label{EQ:strictlyErg-corrEntropy_mu-proof2}
    \frac{-\log \tilde\mu(n)}{n}
    \le
    \frac{4\log l_j + 4\log 2}{l_j}
    \,.
  \end{equation}
  If $n\ge (r_j/2)l_j$ then
  $\mu\big([w]\big)\ge \mu\big([\bar{w}_{j+1}]\big) \ge 1/(2l_{j+1}^2)\ge 1/(2^3 m_j^2l_j^2)$.
  Thus
  \begin{equation}\label{EQ:strictlyErg-corrEntropy_mu-proof3}
    \frac{-\log \tilde\mu(n)}{n}
    \le
    \frac{4\log m_j + 4\log l_j + 6\log 2}{(r_j/2)l_j}
    =
    \frac{8\lambda_j}{r_j}
    + \frac{8\log l_j + 12\log 2}{r_jl_j}
    \,.
  \end{equation}
  Since the right-hand sides of  (\ref{EQ:strictlyErg-corrEntropy_mu-proof2})
  and (\ref{EQ:strictlyErg-corrEntropy_mu-proof3}) converge to zero for $j\to\infty$,
  we have that $\lim_n (-1/n)\log \tilde\mu(n)=0$.
  So $\entr(\sigma,\mu)=0$ by Lemma~\ref{L:subshifts-correl-entropy-measure}.
\end{proof}

Since the subshift $(X,\sigma)$ is strictly ergodic, Proposition~\ref{P:strong-law-uniquelyErgodic}
immediately implies the following result, which finishes the proof of Theorem~\refThmC{}.

\begin{lemma}\label{L:strictlyErg-local_corr_entropy}
  The local correlation entropy $\entr(\sigma,y)$ of every $y\in X$ is zero.
\end{lemma}

\section*{Acknowledgments}
Substantive feedback from Marek \v Spitalsk\'y, Jana Majerov\'a, and Marian Grend\'ar
is gratefully acknowledged.
The author is indebted to Xiaojiang Ye for providing a counterexample to
Lemma~\ref{L:graph-Voptimal-1} in the previous version of the paper.
This research is an outgrowth of the project ``SPAMIA'', M\v S SR-3709/2010-11,
supported by the Ministry of Education, Science, Research and Sport of the Slovak Republic,
under the heading of the state budget support for research and development.
The author also acknowledges support from VEGA~1/0786/15 and APVV-15-0439 grants.

\providecommand{\href}[2]{#2}
\providecommand{\arxiv}[1]{\href{http://arxiv.org/abs/#1}{arXiv:#1}}
\providecommand{\url}[1]{\texttt{#1}}
\providecommand{\urlprefix}{URL }
\providecommand{\newblock}{}
\providecommand{\doititle}[1]{{#1}}

\end{document}